\documentclass[a4paper,twoside, 
 10pt]{amsart}
\usepackage{amssymb,amsfonts,amsmath,
} 
\newtheorem{theor}{~~~~Theorem}

\newtheorem{prop}{~~~~Proposition}
\newtheorem{cor}{~~~~Corollary}
\newtheorem{lemma}{~~~~Lemma}

\allowdisplaybreaks
\newtheorem{remark}{~~~~Remark}
\newtheorem{defin}{~~~~Definition}

\def\og{\leavevmode\raise.3ex\hbox{$\scriptscriptstyle\langle\!\langle$~}}
\def\fg{\leavevmode\raise.3ex\hbox{~$\!\scriptscriptstyle\,\rangle\!\rangle$}}
\begin{document}
\bibliographystyle{plain}
\title{Jacobi Equations and Comparison Theorems for Corank 1
sub-Riemannian Structures with Symmetries}
\date{\today}
\author{
Chengbo Li \address{S.I.S.S.A., Via Beirut 2-4, 34014, Trieste,
Italy; email: chengbo@sissa.it} \and
Igor Zelenko
\address{ Department of Math., Texas A\&M Univ., College Station, TX 77843-3368, USA; email:
zelenko@math.tamu.edu }}

\begin{abstract}
 The Jacobi curve of an extremal of optimal control problem is a curve in a Lagrangian Grassmannian defined up to a symplectic transformation and containing all information about the solutions of the Jacobi equations along this extremal.
In our previous works we constructed the canonical bundle of moving frames and the complete system of symplectic invariants, called curvature maps, for parametrized curves in Lagrange Grassmannians satisfying very general assumptions.
The structural equation for a canonical moving frame of the Jacobi curve of an extremal can be interpreted as the normal form for the Jacobi equation along this extremal and the curvature maps can be seen as the \textquotedblleft coefficients\textquotedblright of  this normal form. In the case of a Riemannian metric there is only one curvature map and it is naturally related to the Riemannian sectional curvature. In the present paper we study the curvature maps for a sub-Riemannian structure on a corank 1 distribution having an additional transversal infinitesimal symmetry. After the factorization by the integral foliation of this symmetry, such sub-Riemannian structure can be reduced to a Riemannian manifold equipped with a closed $2$-form (a magnetic field).
We obtain explicit expressions for the curvature maps of the original sub-Riemannian structure in terms of the curvature tensor of this Riemannian manifold  and the magnetic field. We also estimate the number of conjugate points along the sub-Riemannian extremals in terms of the bounds for the curvature tensor of this Riemannian manifold  and the magnetic field  in the case of an uniform magnetic field. The language developed for the calculation of the curvature maps can be applied to more general sub-Riemannian structures with symmetries, including sub-Riemmannian structures appearing naturally in Yang-Mills fields.
\end{abstract}
\maketitle \markboth {Chengbo Li and Igor
Zelenko} {Jacobi Equations and Comparison Theorems for Corank 1
sub-Riemannian Structures with Symmetries}

\section{Introduction}
\setcounter{equation}{0}
Let $\mathcal D$ be a vector distribution on a manifold $M$, i.e., a subbundle of the tangent bundle $TM$. Assume that an Euclidean structure $\left\langle\cdot,\cdot\right\rangle_q$ is given on each space $\mathcal D_q$ smoothly w.r.t. $q$. The triple $(M, \mathcal D, \left\langle\cdot,\cdot\right\rangle)$ defines a sub-Riemannian structure on $M$. Assume that $M$ is connected and that $\mathcal{D}$ is completely nonholonomic. A Lipschitzian curve $\gamma: [0, T]\longrightarrow M$ is called \emph{admissible} if $\dot\gamma(t)\in \mathcal{D}_{\gamma(t)}$, for a.e. $t$. It follows from the Rashevskii-Chow theorem that any two points in $M$ can be connected by an admissible curve. One can define the length of an admissible curve $\gamma: [0, T]\longrightarrow M$ by $\int_0^T \|\dot\gamma(t)\|dt,$ where $\|\dot\gamma(t)\|=\left\langle\dot\gamma(t),\dot\gamma(t)\right\rangle^{\frac{1}{2}}.$


\subsection{Sub-Riemannian geodesics}
The length minimizing problem is to find the shortest admissible curve connecting two given points on $M$. As in Riemannian geometry, it is equivalent to the problem of minimizing the kinetic energy $\frac{1}{2}\int_0^T \|\dot\gamma(t)\|^2dt$. Indeed, by Schwartz inequality any curve minimizing the kinetic energy is the shortest one and, conversely, an appropriate reparametrization of a shortest curve is an energy minimizer.

The problem can be regarded as an optimal control problem and its extremals can be described by the Pontryagin Maximum Principle of Optimal Control Theory (\cite{pbgmthe}). There are two different types of extremals: abnormal and normal, according to vanishing or nonvanishing of Lagrange multiplier near the functional, respectively. Sub-Riemannian energy (length) minimizers are the projections of either normal extremals or abnormal extremals.

In the present paper we will focus on normal extremals only. To describe them let us introduce some notations.
Let $T^*M$ be the cotangent bundle of $M$ and $\sigma$ be the canonical symplectic form on $T^*M$, i.e., $\sigma=-d\varsigma$, where $\varsigma$ is the tautological (Liouville) 1-form on $T^*M$. For each function $H:T^*M\to \mathbb R$, the Hamiltonian vector field $\vec h$ is defined by $i_{\vec h}\sigma=dh.$
Given a vector $u\in T_qM$ and a covector $p\in T_q^*M$ we denote by $p\cdot u$ the value of $p$ at $u$.
Let
\begin{equation}\label{h}
h(\lambda)\stackrel{\Delta}{=}\max_{u\in\mathcal{D}}(p\cdot u-\frac{1}{2}\|u\|^2)=\frac{1}{2}\|p|_{\mathcal{D}_q}\|^2,\quad\lambda=(p,q)\in T^*M,\ q\in M,\ p\in T^*_qM,
\end{equation}
where $p|_{\mathcal{D}_q}$ is the restriction of the linear functional $p$ to $\mathcal{D}_q$ and the norm $\|p|_{\mathcal{D}_q}\|$ is defined w.r.t. the Euclidean structure on $\mathcal D_q.$  The normal extremals are exactly the trajectories of $\dot\lambda(t)=\vec h(\lambda)$.
\subsection{Jacobi curve and conjugate points along normal extremals}
Let us fix the level set of the Hamiltonian function $h$:
$$\mathcal{H}_{c}\stackrel{\Delta}{=}\{\lambda\in T^*M| h(\lambda)=c\}, c>0$$
Let $\Pi_{\lambda}$ be the vertical subspace of
$T_{\lambda}\mathcal{H}_{c}$, i.e.,
$$
\Pi_{\lambda}=\{\xi\in T_{\lambda}\mathcal{H}_c:
\pi_*(\xi)=0\},
$$
where $\pi: T^*M\longrightarrow M$ is the canonical projection.
With any normal extremal $\lambda(\cdot)$ on $\mathcal H_{c}$, one can associate a curve in a Lagrange Grassmannian which describe the dynamics of the vertical subspaces $\Pi_\lambda$ along this extremal w.r.t. the flow $e^{t\vec h}$, generated by $\vec h$. For this let
\begin{equation}\label{Jacobi}
t\longmapsto \mathfrak J_{\lambda}(t)\stackrel{\Delta} {=}e_*^{-t\vec
h}(\Pi_{e^{t\vec h}\lambda})/\{\mathbb{R}\vec h(\lambda)\}.
\end{equation}
The curve $\mathfrak J_\lambda(t)$ is the curve in the Lagrange Grassmannian of the linear symplectic space
$W_\lambda = T_\lambda\mathcal H_{c}/{\mathbb R\vec h(\lambda)}$ (endowed with the symplectic form induced in the obvious way by the canonical symplectic form $\sigma$ of $T^*M$). It is called the \emph{Jacobi curve} of the extremal $e^{t\vec h}\lambda$ (attached at the point $\lambda$).

The reason to introduce Jacobi curves is two-fold. On one hand, it can be used to construct differential invariants of sub-Riemannian structures, namely, any symplectic invariant of Jacobi curve, i.e., invariant of the action of the linear symplectic group $Sp(W_\lambda)$ on the Lagrange Grassmannian $L(W_\lambda)$, produces an invariant of the original sub-Riemannian structure. On the other hand, the Jacobi curve contains all information about conjugate points along the extremals. Then a natural question arises: how do the symplectic invariants effect the appearance of the conjugate points?

Recall that time $t_0$ is called conjugate to $0$ if
\begin{equation}\label{conju}
e^{t_0\vec h}_*\Pi_\lambda\cap\Pi_{e^{t_0\vec h}\lambda}\neq 0.
\end{equation}
and the dimension of this intersection is called the multiplicity of $t_0$.
The curve  $\pi(\lambda(\cdot))|_{[0, t]}$ is $W^1_\infty$-optimal (and even $C$-optimal) if there is no conjugate point in $(0, t)$ and is not optimal otherwise. Note that (\ref{conju}) can be rewritten as: $e^{-t_0\vec h}_*\Pi_{e^{t_0\vec h}\lambda}\cap \Pi_\lambda\neq 0$, which is equivalent to $$\mathfrak J_\lambda(t_0)\cap \mathfrak J_\lambda(0)\neq 0.$$

\subsection{Statement of the problem}
In our previous papers (\cite{icdifferential}, \cite{icparametrized}), we constructed the canonical bundle of moving frames and the complete system of symplectic invariants for parametrized curves in Lagrange Grassmannians satisfying very general assumptions.
As a consequence, for any sub-Riemannian structure defined on any nonholonomic distribution on a manifold $M$ one has the canonical (in general, non-linear) connection on an open subset of the cotangent bundle, the canonical splitting of the tangent spaces to the fibers of the cotangent bundle and the tuple of maps, called curvature maps, between the subspaces of the splitting intrinsically related to the sub-Riemannian structure.  We give a brief description of these constructions in section 2.
The structural equation for a canonical moving frame of the Jacobi curve of an extremal can be interpreted as the normal form for the Jacobi equation along this extremal and the curvature maps can be seen as the \textquotedblleft coefficients\textquotedblright of  this normal form.
In the case of a Riemannian metric the canonical connection above coincides with the Levi-Civita connection and  the splitting of the tangent spaces to the fibers is trivial. Moreover, there is only one curvature map and it is naturally related to the Riemannian sectional curvature tensor.

However, for the  proper sub-Riemannian structures (i.e. when $D\neq TM$) very little is known about the curvature maps, except that they depend rationally on points of fibers of $T^*M$.
In order to interpret better these invariants, to understand their role in optimality properties of sub-Riemannian extremals and other qualitative properties of flows of extremals,
we suggest to study them for a special class of sub-Riemannian metrics having
sufficiently many symmetries such that after an appropriate number of factorizations one gets a Riemannian metric.
Such sub-Riemannian structures appear naturally on principal connections of principal bundles over Riemannian manifolds (including Yang-Mills fields as a particular case): the sub-Riemannian structure is given by a pull-back (with respect to the canonical projection) of the Riemannian metric of the base manifold to the distribution defining the connection.

How the above-mentioned curvature maps are expressed in terms of the Riemannian curvature tensor of the base manifold and the curvature form of the principal connection? How to estimate the number of conjugate points in terms of the bounds of the Riemannian curvature tensor of the base manifold and the curvature form of the principal connection? We answer these questions in the case when principal bundles have one-dimensional fibers. It is well known that such geometric structures describe magnetic fields on Riemannian manifolds, where the connection form is seen as the magnetic potential.
 The main results of the paper are the explicit expressions of the curvature maps (Theorems \ref{main1}-\ref{main3} below) and the estimation of the number of conjugate points along sub-Riemannian extremals (Theorem \ref{comparison1} below) in terms of the Riemannian curvature tensor of the base manifold and the magnetic field (the latter is done in the case of the uniform magnetic field). 
We also believe that the coordinate-free language we introduced  in sections 3 and 4 for calculation of these invariants will
be useful in the treatment of the more general situations mentioned above.
\section{Differential geometry of curves in Lagrange Grassmannian}
\setcounter{equation}{0}
\setcounter{theor}{0}
\setcounter{lemma}{0}
\setcounter{prop}{0}

In this section we briefly describe the construction of the above-mentioned
curvature maps. The details can be found in \cite{icdifferential}, \cite{icparametrized}. Denote by $L(W)$ the Lagrangian Grassmannian of an even dimensional linear symplectic space $W$ endowed with a symplectic form $\omega$.
Given $\Lambda\in L(W)$, the tangent space $T_{\Lambda}L(W)$ of $L(W)$ at point $\Lambda$ can be naturally identified with the space $Quad(\Lambda)$ of all quadratic forms on linear space $\Lambda\subset W$.
A curve $\Lambda(\cdot)$ is called \emph{monotonically nondecreasing} (\emph{monotonically nonincreasing}) if the velocity is
nonnegative definite (nonpositve definite) at any point.
\subsection{Young diagrams}
Denote by $C(\Lambda)$ the canonical bundle over $\Lambda$: the fiber of $C(\Lambda)$ over the point $\Lambda(t)$ is the linear space $\Lambda(t)$. Let $\Gamma(\Lambda)$ be the space of sections of $C(\Lambda)$. Define the $i$th extension of $\Lambda(\cdot)$ (or the $i$-th osculating space) by
$$\Lambda^{(i)}(t)={\rm span}\{\frac{d^j}{d\tau^j}\ell(\tau): \ell(\tau)\in\mathcal{C}(\Lambda), 0\leq j\leq i\}.$$
The flag
$ \Lambda(t)\subseteq\Lambda^{(1)}(t)\subseteq\Lambda^{(2)}(t)\subseteq...$
is called {\it the associated flag of the curve $\Lambda(\cdot)$ at point $t$.} Assume that the following two conditions hold:
\begin{enumerate}
 \item $\hbox{dim}\ \Lambda^{(i)}(t)-\hbox{dim}\ \Lambda^{(i-1)}(t)$ is independent of $t$ for any $i$;
 \item $\Lambda^{(p)}(t)=W$ for some $p\in\mathbb{N}.$
\end{enumerate}
\begin{remark}
 Both of the assumptions are not restrictive: the first holds in a neighborhood of generic point and the second holds after the appropriate factorization.
\end{remark}

It follows from the first assumption above that
$$\hbox{dim}\ \Lambda^{(i+1)}(t)-\hbox{dim}\ \Lambda^{(i)}(t)\leq \hbox{dim}\ \Lambda^{(i)}(t)-\hbox{dim}\ \Lambda^{(i-1)}(t).$$
Therefore, using the flag, to any $\Lambda(\cdot)$ we can assign the Young diagram in the following way:  the number of boxes of the $i$th column is equal to $\dim\ \Lambda^{(i)}(t)-\dim\ \Lambda^{(i-1)}(t)$. Assume that the length of the rows of $D$ be $p_1$ repeated $r_1$ times, $p_2$ repeated $r_2$ times, . . ., $p_d$ repeated $r_d$ times with
$p_1 > p_2 > . . . > p_d$. In this case, the Young diagram $D$ is the union of $d$ rectangular diagrams of size  $r_i\times p_i, 1\leq i\leq d.$ Denote them by $D_i, 1\leq i\leq d.$ The Young diagram $\Delta$, consisting of $d$ rows such that the $i$th row has $p_i$ boxes, is called the reduced diagram or the reduction of the diagram $D$. 
The rows of $\Delta$ will be called levels. To the $j$th box $a$ of the $i$th level of $\Delta$ one can assign the $j$th column of the rectangular subdiagram $D_i$ of $D$ and the integer number $r_i$ (equal to the number of boxes of $D$ in this subcolumn), called the size of the box $a$.

\subsection{Normal moving frames}
As usual, by $\Delta\times \Delta$ we will mean the set of pairs of boxes of $\Delta$. Also denote by ${\rm Mat}$ the set of matrices of all sizes. The mapping $R: \Delta\times \Delta\longrightarrow \rm {Mat}$
is called \emph {compatible with the Young diagram $D$},
 if to any pair $(a, b)$ of boxes of sizes $s_1$ and $s_2$ respectively
the matrix $R(a,b)$ is of the size $s_2\times s_1$. The compatible
mapping $R$ is called \emph{symmetric} if for any  pair $(a,b)$ of
boxes the following identity holds
\begin{equation}
\label{symm} R(b, a)=R(a, b)^T.
\end{equation}
 Denote by $\Upsilon_i$ the $i$th level of $\Delta$.
Also denote  by $a_i$ and $\sigma_i$  the first and the last
boxes of the $i$th level $\Upsilon_i$ respectively and by
$r:\Delta\backslash\{\sigma_i\}_{i=1}^d\longrightarrow \Delta$ the right
shift on the diagram $\Delta$. The last box of any level will be called \emph{special}.
For any pair of integers $(i,j)$ such that
$1\leq j<i\leq d$ consider the following tuple of pairs of boxes
\begin{equation}
\label{chain}
\begin{split}
~&\bigl (a_j,a_i\bigr),\, \bigl(a_j, r(a_i)\bigr),\, \bigl(r(a_j),
r(a_i)\bigr),\,\bigl(r(a_j), r^2(a_i)\bigr), \ldots,
\bigl(r^{p_i-1}(a_j), r^{p_i-1}(a_i)),
\\
~&\bigl (r^{p_i}(a_j), r^{p_i-1}(a_i)),\ldots, \bigl(r^{p_j-1}(a_j),
r^{p_i-1}(a_i)\bigr).
\end{split}
\end{equation}

\begin{defin}
\label{qnormmap} A symmetric compatible mapping $R:\Delta\times
\Delta\longrightarrow \rm {Mat}$ is called \underline{normal}  if the
following three conditions hold: \vskip .1in
\begin{enumerate}
\item For any $1\leq j<i\leq
d$, the matrices, corresponding to the first $(p_j-p_i-1)$ pairs of
the tuple \eqref{chain}, are equal to zero;
\item
Among all matrices $\mathcal R(a,b)$, where the box $b$ is not higher than the box $a$
in the diagram $\Delta$
the only possible nonzero matrices are the following:
the matrices $\mathcal R(a,a)$ for
all $a\in \Delta$, the matrices $R
\bigl(a,r(a)\bigr)$, $R\bigl(r(a),a\bigr)$ for all nonspecial boxes,
and the matrices, corresponding to the pairs, which appear in the
tuples \eqref{chain},
for all $1\leq j<i\leq d$;
\item  The matrix $R\bigl(a,r(a)\bigr)$ is antisymmetric for any nonspecial box  $a$.
\end{enumerate}
\end{defin}

Note that this notion depends only on the mutual locations of the
boxes $a$ and $b$ in the diagram $\Delta$.
Now let us fix some terminology about the frames in $W$, indexed by
the boxes of the Young diagram $D$.
A frame $\bigl(\{e_\alpha\}_{\alpha\in D},\{f_\alpha\}_{\alpha\in
D}\bigr)$ of $W$ is called \emph{Darboux} or {\it symplectic}, if
for any $\alpha,\beta\in D$ the following relations hold
\begin{equation}
\label{Darboux}
\omega(e_\alpha, e_\beta)=0,\quad\omega(f_\alpha, f_\beta)=0,\quad
\omega(e_\alpha, f_\beta)=\delta_{\alpha,\beta},
\end{equation}
where $\delta_{\alpha,\beta}$ is the analogue of the Kronecker index
defined on $D\times D$.
In the sequel it will be  convenient to
divide
a moving frame
$\bigl(\{e_\alpha(t)\}_{\alpha\in D},\{f_\alpha(t)\}_{\alpha\in
D}\bigr)$ of $W$ indexed by the boxes of the Young diagram $D$ into
the tuples of vectors indexed by the boxes of the reduction
$\Delta$ of $D$, according to the correspondence between the
boxes of $\Delta$ and the subcolumns of $D$. More precisely,
given a box $a$ in $\Delta$ of size $s$, take all boxes
$\alpha_1, \ldots, \alpha_s$ of the corresponding subcolumn in $D$
in the order from the top to the bottom and denote
\begin{equation*}
E_a(t)=
\bigl(e_{\alpha_1}(t), \ldots, e_{\alpha_s}(t)\bigr),\quad
F_a(t)=
\bigl(f_{\alpha_1}(t),\ldots, f_{\alpha_s}(t)\bigr).
\end{equation*}
\begin{defin}
\label{normframe} The moving Darboux frame $(\{E_a(t)\}_{a\in
\Delta}, \{F_a(t)\}_{a\in \Delta})$
is called the normal moving frame of a monotonically
nondecreasing curve $\Lambda(t)$ with the Young diagram $D$,
if
$$\Lambda(t)={\rm span}\{E_a(t)\}_{a\in \Delta}$$ for any $t$
and there exists an one-parametric family of normal
mappings $R_t:\Delta\times\Delta\longrightarrow \rm {Mat}$ such that the
moving frame $(\{E_a(t)\}_{a\in \Delta}, \{F_a(t)\}_{a\in \Delta})$
satisfies the following structural equation:
\begin{equation}
\label{structeq}
\begin{cases}
 E_a'(t)=E_{l(a)}(t)& \text {if\,\, $a\in
\Delta\backslash\ \mathcal F_1$}\\
E_a'(t)=F_a(t)& \text {if\,\, $a\in \mathcal F_1$}\\
F_a'(t)=-\sum\limits_{b\in\Delta}E_b(t) R_t(a,b)-F_{r(a)}(t)& \text {if\,\,
$a\in
\Delta\backslash\ \mathcal S$}\\
F_a'(t)=-\sum\limits_{b\in\Delta}E_b(t)R_t(a,b)& \text {if\,\,  $a\in
\mathcal S$}
\end{cases},
\end{equation}
where
$\mathcal F_1$ is the first column of the diagram $\Delta$,
$\mathcal S$ is the set of all its special boxes, and
$l:\Delta\backslash \mathcal F_1\longrightarrow\Delta$, $r:\Delta\backslash\
\mathcal S\longrightarrow \Delta$ are the left and right shifts on the
diagram $\Delta$. The mapping $R_t$, appearing in (\ref{structeq}),
is called the normal mapping, associated with the
normal moving frame $(\{E_a(t)\}_{a\in \Delta}, \{F_a(t)\}_{a\in
\Delta})$.
\end{defin}
\begin{theor}\label{oldtheor}
For any monotonically nondecreasing curve
$\Lambda(t)$ with the Young diagram $D$ in the Lagrange Grassmannian
there exists a normal moving frame
 $(\{E_a(t)\}_{a\in \Delta}, \{F_a(t)\}_{a\in
\Delta})$. A moving frame
$$(\{\widetilde E_a(t)\}_{a\in \Delta}, \{\widetilde F_a(t)\}_{a\in
\Delta})$$ is a normal moving frame of the curve $\Lambda(\cdot)$ if
and only if for any $1\leq i\leq d$ there exists a constant
orthogonal matrix $U_i$ of size $r_i\times r_i$ such that for all
$t$
\begin{equation}
\label{u} \widetilde E_a(t)= E_a(t)U_i,\quad \widetilde F_a(t)=
F_a(t)U_i, \quad \forall \,a\in \Upsilon_i.
\end{equation}
\end{theor}
As a matter of fact, normal moving frames define a principal $O(r_1)\times O(r_2)\times...\times O(r_k)$-bundle of symplectic frame in $W$ endowed with a canonical connection. 
The normal moving frames are horizontal curves of this connection. 

Relations (\ref{u}) imply that for any box $a\in\Delta$ of size
$s$ the following $s$-dimensional subspaces
\begin{equation}
\label{V} V_a(t)={\rm span}\{E_a(t)\},\ V_a^{\hbox{trans}}(t)={\rm span}\{F_a(t)\}
\end{equation}
of $\Lambda(t)$ does not depend on the choice of the normal moving
frame. 
In particular, there exists the \emph {canonical splitting of the subspace
$\Lambda(t)$} defined by
\begin{equation}
\label{cansplit} \Lambda(t)=\bigoplus_{a\in \Delta} V_a(t),\ \hbox{dim}\ V_a(t)=\hbox{size}(a)
\end{equation}
and the \emph {canonical complement $\Lambda^{\rm trans}(t)$} to $\Lambda(t)$ defined by
\begin{equation}
\label{cancomp}
\Lambda^{\rm trans}(t)=\bigoplus_{a\in \Delta}V^{\rm trans} _a(t).
\end{equation}
Moreover, each subspace $V_a(t) (\hbox{and}\ V_a^{\hbox{trans}}(t))$ is endowed with the \emph
{canonical Euclidean structure} such that the tuple of vectors $E_a (\hbox{and}\ F_a(t))$
constitute an orthonormal  frame w.r.t. to it. Taking the canonical Euclidean structures on all $V_a(t)$ and assuming that
subspaces $V_a(t)$ and $V_b(t)$ with different $a$ and $b$ are orthogonal, we get the canonical Euclidean structure on the whole $\Lambda(t)$.
The linear map from $V_a(t)$ to $V_b(t)$ with the matrix $R_t(a, b)$ from (\ref{structeq}) in the basis $\{E_a(t)\}$ and $\{E_b(t)\}$ of $V_a(t)$ and $V_b(t)$ respectively, is independent of the choice of normal moving frames. It will be denoted by $\mathfrak R_t(a, b)$ and it is called the \emph {$(a,b)$-curvature map of the curve $\Lambda(\cdot)$ at time $t$}. Finally, all $(a,b)$-curvature maps form the canonical map $\mathfrak R_t:\Lambda(t)\rightarrow \Lambda(t)$ as follows:
\begin{equation}
\label{bigr}
R_t v_a=\sum_{b\in\Delta} R_t(a,b)v_a, \forall v_a\in V_a(t), a\in\Delta.
\end{equation}
The map $\mathfrak R_t$ is called the \emph{big curvature map of the curve $\Lambda(\cdot)$ at time $t$}.

\subsection{Consequences for sub-Riemannian Structures}\label{conse}
Let $(M, \mathcal{D}, \left\langle\cdot, \cdot\right\rangle)$ be a sub-Riemannian structure. Note that the Jacobi curve associated with an extremal in $M$ is monotonically nondecreasing. A point $\lambda\in T^*M$ is called a $D$-regular point if the germ of the Jacobi curve $\mathfrak J_\lambda(t)$ at $t=0$ has the Young diagram $D$. Assume that for some diagram $D$ the set of $D$-regular point is open in $\mathcal H_{\frac{1}{2}}$ and let $\Delta$ be the reduced diagram of $D$. The structural equation (\ref{structeq}) for the Jacobi curve $\mathfrak J_\lambda(t)$ can be seen as the intrinsic Jacobi equation along the extremal $e^{t\vec h}\lambda$ and the $(a, b)-$curvature maps are the coefficients of this Jacobi equation.

Since there is a canonical splitting of $\mathfrak J_\lambda(t)$ and taking into account that $\mathfrak J_\lambda(0)$ and $\Pi_\lambda$ can be naturally identified, we have the canonical splitting of $\Pi_\lambda$:
$$\Pi_\lambda=\bigoplus_{a\in\Delta}\mathcal V_a(\lambda),\ \hbox{dim}\ (\mathcal V_a(\lambda))=\hbox{size}(a),$$
where $\mathcal V_a(\lambda)=V_a(0)$.

Moreover, let $\mathfrak R_\lambda(a,b): \mathcal V_a(\lambda)\rightarrow \mathcal V_b(\lambda)$  and the $\mathfrak R_\lambda:\Pi_\lambda\rightarrow\Pi_\lambda$ be the $(a,b)$-curvature map and the big curvature of the Jacobi curve $\mathfrak J_\lambda(\cdot)$ at $t=0$.
These maps are intrinsically related to the sub-Riemannian structure.
They are called the \emph{(a,b)-curvature} and the \emph{big curvature} of the sub-Riemannian structure at the point $\lambda$.
Also, the canonical complement $\mathfrak J^{\hbox{trans}}_\lambda
(t)$ at $t=0$ give rise a canonical complement of $\Pi_\lambda$ in $W_\lambda$, where $W_\lambda=T_\lambda\mathcal H_{\frac{1}{2}}/\mathbb R\vec h$, as before. For any $a\in\Delta$, denote

\begin{equation}
\label{vtrans}
\mathcal V_a^{\hbox{trans}}(\lambda)=V_a^{\hbox{trans}}(0).
\end{equation}
It turns out that $\displaystyle{\bigoplus_{a\in\Delta}}\mathcal V_a^{\hbox{trans}}(\lambda)\oplus\mathbb R\vec h$ defines \emph {the canonical (non-linear) connection of $T^*M$}.

Let $\lambda\in T^*M$ and let $\lambda(t)=e^{t\vec h}\lambda$.
Assume that $(E_a^\lambda(t),F_a^\lambda(t))_{a\in\Delta}$ is a normal moving frame of the Jacobi curve $\mathfrak J_\lambda(t)$ attached at point $\lambda$.
Let $\mathfrak E$ be the Euler field on $T^*M$, i.e. the infinitesimal generator of the homotheties  on its fibers. Clearly
$T_\lambda(T^*M)=T_\lambda\mathcal H_{h^{-1}(\lambda)}\oplus\mathbb R \mathfrak E(\lambda)$.
The flow $e^{t\vec h}$ on $T^*M$ induces the push-forward maps $\bigl(e^{t\vec h}\bigr)_*$ between the corresponding tangent spaces $T_\lambda T^*M$ and $T_{e^{t\vec h}\lambda}T^*M$, which in turn induce naturally the maps between the spaces $T_\lambda(T^*M)/\mathbb R\vec h(\lambda)$ and $T_{e^{t\vec h}\lambda}T^*M/\mathbb R\vec h(e^{t\vec h}\lambda)$.
The map $\mathcal K^t$ between $T_\lambda(T^*M)/\mathbb R\vec h(\lambda)$ and $T_{e^{t\vec h}\lambda}T^*M/\mathbb R\vec h(e^{t\vec h}\lambda)$,
sending $E_a^\lambda(0)$ to $\bigl(e^{t\vec h}\bigr))_*E_a^{\lambda}(t)$, $F_a^\lambda(0)$ to $\bigl(e^{t\vec h}\bigr)_*F_a^{\lambda}(t)$ for any $a\in\Delta$, and the equivalence class of $\mathfrak E(\lambda)$ to the equivalence class of
$\mathfrak E(e^{t\vec h}\lambda)$, is independent of the choice of normal moving frames. The map $\mathcal K^t$ is called \emph{the parallel transport} along the extremal $e^{t\vec h}\lambda$ at time $t$.  For any $v\in T_\lambda(T^*M)/\mathbb R\vec h(\lambda)$, its image $v(t)=\mathcal K^t(v)$ is called \emph{ the parallel transport of $v$ at time $t$}. Note that from the definition of the Jacobi curves and the construction  of normal moving frame it follows that the restriction of the parallel transport $\mathcal K_t$ to the vertical subspace $T_\lambda(T_{\pi(\lambda)}^*M)$ of $T_\lambda(T^*M)$ can be considered as a map onto the vertical subspace $T_{e^{t\vec h}\lambda}(T_{\pi(e^{t\vec h}\lambda)}^*M)$ of $T_{e^{t\vec h}\lambda}(T^*M)$. A vertical vector field $V$ is called \emph{parallel} if $V(e^{t\vec h}\lambda)=\mathcal K^t\bigl(V(\lambda)\bigr)$.


In the \textit{Riemannian case}, i.e., when $\mathcal{D}=TM$, the Young diagram of the Jacobi curve $\Lambda(\cdot)$ consists of only one column and the corresponding reduced diagram consists of only one box. Denote this box by $a$. The structure equation for a normal moving frame is of the form:
 \begin{equation}
 \label{structRiem}
 \left\{\begin{array}{l}E_a'(t)=F_a(t)\\
 F_a'(t)=-E_a(t)\mathcal R_t(a,a).
 \end{array}\right.
 \end{equation}
\begin{remark}
 \label{lcrem}
 Note that from \eqref{structRiem} it follows that if $\bigl(\widetilde E_a(t), \widetilde F_a(t)\bigr)$ is a Darboux moving frame such that
 $\widetilde E_a(t)$ is an orthonormal frame of $\Lambda(t)$  and ${\rm span}\,\{\widetilde F_a(t)\}=\Lambda^{\rm trans}(t)$. Then there exists a curve of antisymmetric matrices $B(t)$ such that
\begin{equation}
 \label{structRiem1}
\left\{\begin{array}{l}
 \widetilde E_a'(t)=\widetilde E_a(t) B(t)+\widetilde F_a(t)\\
 \widetilde F_a'(t)=-\widetilde E_a(t)\widetilde{\mathcal R}_t(a,a)+\widetilde F_a(t) B(t),
\end{array}\right.
\end{equation}
where $\widetilde{\mathcal R}_t(a,a)$ is the matrix of the curvature map $\mathfrak R_t(a,a)$ on $\Lambda(t)$ w.r.t. the basis $\widetilde E_a(t)$.
\end{remark}

 In \cite{agfeedback} and \cite{ageometry} it was shown that  in the considered case the canonical connection coincides with the Levi-Civita connection and the unique curvature map $\mathfrak R_\lambda(a, a): \mathcal V_a(\lambda)\longrightarrow\mathcal V_a(\lambda)$ (where $\mathcal V_a(\lambda)=\Pi_\lambda$) was expressed by the Riemannian curvature tensor. In order to give this expression let  $R^\nabla$ be the Riemannian curvature tensor. Below we will use the identification between the tangent vectors and the cotangent vectors of the Riemannian manifold $M$ given by the Riemannian metric. More precisely, given $p\in T_q^*M$ let $p^h\in T_qM$ such that $p\cdot v=\langle p^h,v\rangle$ for any $v\in T_q M$.  Since tangent spaces to a linear space at any point are naturally identified with the linear space itself we can also identify in the same way the space $T_\lambda(T_{\pi(\lambda)}^*M)$ with $T_{\pi(\lambda)}M$.
\begin{equation}\label{Rie}
\mathfrak R_{\lambda}(a, a)v=R^\nabla(p^h, v^h)p^h,\quad\forall \lambda=(q, p)\in \mathcal H_{h^{-1}(\lambda)} ,q\in M, p\in T^*_qM, \quad v \in \Pi_\lambda.
\end{equation}
Given a vector $X\in T_qM$ denote by $\nabla_X$ its lift to the Levi-Civita connection, considered as an Ehresmann connection on $T^*M$. Then by constructions the Hamiltonian vector field $\vec h$ is horizontal and satisfies $\vec h=\nabla_p$. Take any
$v, w \in \Pi_\lambda$ and let $V$ be a vertical vector field such that $V(\lambda)=v$.
From \eqref{Rie} , structure equation \eqref{structRiem}, and the fact that the Levi-Civita connection (as an Ehresmann connection on $T^*M$) is a Lagrangian distribution it follows that
the Riemannian curvature tensor satisfies the following identity:
\begin{equation}
\label{Riemnabla}
\langle R^\nabla(p^h, v^h)p^h, w^h\rangle=-\sigma\left([\nabla_{p^h},\nabla_{V^h}](\lambda),\nabla _{w^h}\right).
\end{equation}

For the nontrivial case of sub-Riemannian structures, i.e., when $\mathcal D\subsetneqq TM$, let us consider the simplest case: the sub-Riemannian structure on a nonholonomic corank 1 distribution. Fix $\hbox{dim}\ M=n(n\geq 3).$ Recall that our considerations are local, thus we can select a nonzero 1-form $\omega_0$ satisfying $\omega_0|_{\mathcal D}=0.$
Then $d\omega_0|_\mathcal D$ is well-defined nonzero $2$-form up to a multiplication of nonzero function. Therefore, for any $q\in M$, the skew-symmetric linear map $J_q: \mathcal D_q\longrightarrow\mathcal D_q$ satisfying $d\omega_0(q)(X, Y)=\left\langle J_qX, Y\right\rangle_q, \forall X, Y\in \mathcal D_q$ is well-defined up a nonzero
constant. Let
$$\mathcal D^\bot=\{(p, q)\in T^*M: p\cdot v=0,\ \forall v\in\mathcal D_q\},\ \mathcal D_q^\bot=\mathcal D^\bot\cap T^*_qM.$$
Besides, one has the following series of natural identifications:
\begin{equation}\label{ident1}
T^*_qM/\mathcal{D}_q^\bot\sim \mathcal{D}^*_q\stackrel{\left\langle\cdot, \cdot\right\rangle}{\sim}\mathcal{D}_q,
\end{equation}
where $\mathcal D_q^*\subseteq T^*_qM$ is the dual space of $\mathcal D_q$.
According to this identification, $J_q$ can be taken as the linear map from the fiber $T_q^*M$ of $T^*M$ to $T^*_qM/\mathcal D_q^\bot$ (in this case, $J_q|_{\mathcal D_q^\bot}=0	$).

Let $D$ be the Young diagram consisting of two columns, with $(n-2)$ boxes in the first column and 1 box in the second column. Then the set of $D$-regular points coincides with $\{(p, q)\in\mathcal H_{\frac{1}{2}}: J_qp\neq 0\}$(see step 1 of subsection 3.3  Proposition \ref{regular} below for the proof in the particular case with symmetries) . In the case of $n>3$, the reduced Young diagram consists of three boxes: two in the first column and one in the second. The box in the second column will be denoted by $a$, the upper box in the first column will be denoted by $b$ and the lower box in the first column will be denoted by $c$. Note that $\hbox{size}(a)=\hbox{size}(b)=1$ and $\hbox{size}(c)=n-3.$ When $n=3$, the reduced Young diagram consists of two boxes, $a$ and $b$ as above and the box $c$ doesn't appear. All formulae for $n>3$ will be true for $n=3$ if one avoids the formulae containing the box $c$. In this case, the symmetric (Darboux) compatible mapping (with Young diagram $D$) is normal if and only if $R_t(a, b)=0$
and the canonical splitting of $\Pi_\lambda$ has the form: $\Pi_\lambda=\mathcal V_a(\lambda)\oplus \mathcal V_b(\lambda)\oplus \mathcal V_c(\lambda)$, where $\mathcal V_a(\lambda), \mathcal V_b(\lambda)$ are of dimension 1 and $\mathcal V_c(\lambda)$ is of dimension $n-3$. These subspaces can be described as follows. As the tangent space of the fibers of $T^*M$ can be naturally identified with the fibers themselves (the fibers are linear spaces), one can show that $$\mathcal V_a(\lambda)=\mathcal{D}_{\pi(\lambda)}^\bot.$$ Using the fact that $\mathcal V_b(\lambda)\oplus \mathcal V_c(\lambda)\oplus\mathbb{R}p$ is transversal to $\mathcal D_q^\bot$, one can get the following identification
\begin{equation}\label{ident2}
\mathcal V_b(\lambda)\oplus \mathcal V_c(\lambda)\oplus\mathbb{R}p\sim T^*_qM/\mathcal{D}_q^\bot,
\end{equation}
Finally, combining (\ref{ident1}) and (\ref{ident2}), we have that
\begin{equation}\label{ident3}
\mathcal V_b(\lambda)\oplus \mathcal V_c(\lambda)\oplus\mathbb{R}p\sim\mathcal {D}_q^*\sim \mathcal{D}_q,
\end{equation}
Under the identifications, one can show that (see step 1 in subsection \ref{Imp} below):
\begin{equation}\label{ident4}
\mathcal V_b(\lambda)=\mathbb{R}J_qp,\quad \mathcal V_c(\lambda)=({\rm span}\{p, Jp\})^\bot.
\end{equation}



Regarding the $(a, b)-$curvature maps, even in the considered case it is difficult to get the explicit expression in terms of sub-Riemannian structures without additional assumptions. Here we calculate them in the special case of sub-Riemannian structures on corank 1 distribution, having additional infinitesimal symmetries. After an appropriate factorization, such structure can be reduced to a Riemannian manifold equipped with a symplectic form (a magnetic field) and the curvature maps can be expressed in terms of the Riemannian curvature tensor and the magnetic field.


\section{Algorithm for calculation of canonical splitting and $(a,b)$-curvature maps}
\setcounter{equation}{0}
\setcounter{theor}{0}
\setcounter{lemma}{0}
\setcounter{prop}{0}

We begin with the discussion of sub-Riemannian structures with additional symmetries and show that they can be reduced to a Riemannian manifold with a symplectic form. Then we describe the algorithm of finding of normal moving frames for the Jacobi curves of the extremals of such structures. As a result, we write down the canonical complement $\mathcal V^{\hbox{trans}}(\lambda)$ using the symplectic form $\sigma$,  Lie derivatives w.r.t. $\vec h$ and the tensor $J$. Further, we establish certain calculus relating Lie derivatives and the covariant derivative of the reduced Riemannian structure. As a result, we can characterized sub-Riemannian connection in terms of Levi-Civita connection and the tensor $J$.

\subsection{Corank 1 sub-Riemannian structures with symmetries}\label{corank1}
As before, assume that $\mathcal D$ is a nonholonomic corank 1 distribution. Assume that the sub-Riemannian structure $(M, \mathcal{D}, \left\langle\cdot, \cdot\right\rangle)$ has an additional infinitesimal symmetry, i.e., a vector field $X_0$ such that
$$e^{tX_0}_*\mathcal{D}=\mathcal{D}\ ,\  (e^{tX_0})^*\left\langle\cdot, \cdot\right\rangle=\left\langle\cdot, \cdot\right\rangle.$$
Assume also that $X_0$ is transversal to the distribution $\mathcal{D}$,
$\mathbb{R}X_0\oplus\mathcal{D}_q=T_qM, \forall q\in M.$
In this case, the $1-$form $\omega_0$, defined by $\omega_0|_{\mathcal D}=0$, as before, can be determined uniquely by imposing the condition $\omega_0(X_0)=1.$ Therefore $d\omega_0|_{\mathcal D}$ and the operator $J_q$ are also determined uniquely. Let $\xi$ be the 1-foliation generated by $X_0$. Denote by $\widetilde M$ the quotient of $M$ by the leaves of $\xi$ and denote the factorization map by $\hbox{pr}: M\longrightarrow\widetilde M$. Since our construction is local, we can assume that $\widetilde M$ is a manifold. The sub-Riemannian metric $\left\langle\cdot, \cdot\right\rangle$ induces a Riemannian metric $g$ on $\widetilde M$. Also $d\omega_0$ and $J_q$ induce a symplectic form $\Omega$ and a type $(1, 1)$ tensor on $\widetilde M$, respectively. We denote the $(1, 1)$ tensor by $J$ as well. Actually, $\Omega$ can be seen as a magnetic field and $J$ can be seen as a Lorenzian force on Riemannian manifold $\widetilde M$. The projection by ${\rm pr}$ of all sub-Riemannian geodesics describes all possible motion of a charged particle (with any possible charge) given by the magnetic field $\Omega$ on the Riemannian manifold $\widetilde M$(see e.g. \cite[Chapter 12]{matour} and the references therein). 

Define $u_0: T^*M\longrightarrow \mathbb R$ by $u_0(p,q)\stackrel{\Delta}{=}p\cdot X_0(q),\ (p,q)\in T^*M, q\in M,p\in T^*_qM.$
Since $X_0$ is a symmetry of the sub-Riemannian structure, the function $u_0$ is the first integral of the extremal flow, i.e., $\{h, u_0\}=0$, where $\{\cdot,\cdot\}$ is the Poisson bracket.

\subsection{Algorithm of normalization}\label{algorithm}
%

First let us describe the construction of the normal moving frames  and the curvature maps for a monotonically nondecreasing curve $\Lambda(t)$ with the Young diagram $D$ as in subsection \ref{conse}. The details can be found in \cite{icdifferential}.
In this case, the structural equation for the normal moving frame is of the form:
\begin{equation}
\label{structeq1}
\begin{cases}
E_a^\prime(t)=E_b(t)\\
E_b^\prime(t)=F_b(t)\\
E_c^\prime(t)=F_c(t)\\
F_a^\prime(t)=-E_c(t)R_t(a,c)-E_a(t)R_t(a, a)\\
F_b^\prime(t)=-E_c(t)R_t(b,c)-E_b(t)R_t(b, b)-F_a(t)\\
F_c^\prime(t)=-E_c(t)R_t(c, c)-E_b(t)R_t(c,b)-E_a(t)\mathcal{R}_t(c,a).
\end{cases}
\end{equation}

 Assume that each element of the set $\{\mathcal E_a(\lambda),\mathcal E_b(\lambda),\mathcal E_c(\lambda),\mathcal F_a(\lambda),\mathcal F_b(\lambda),\mathcal F_c(\lambda)\}$ is either a vector field or  a tuple of vector fields, depending on the size of the corresponding box in the Young diagram such that
\begin{eqnarray*}
&&(\mathcal E_a(e^{t\vec h}\lambda),\mathcal E_b(e^{t\vec h}\lambda),\mathcal E_c(e^{t\vec h}\lambda),\mathcal F_a(e^{t\vec h}\lambda),\mathcal F_b(e^{t\vec h}\lambda),\mathcal F_c(e^{t\vec h}\lambda))\\
&=&\mathcal K^t(\mathcal E_a(\lambda),\mathcal E_b(\lambda),\mathcal E_c(\lambda),\mathcal F_a(\lambda),\mathcal F_b(\lambda),\mathcal F_c(\lambda)),
\end{eqnarray*}
where $\mathcal K^t$ is the parallel transport, defined in subsection 2.3.
Recall that for any vector fields $X, Y$ one has the following formula:
$\frac{d}{dt}\left.{\!\!\frac{}{}}\right|_{t=0}e^{-t X}_*Y=\hbox{ad}_XY$. So, the derivative w.r.t. $t$ on the level of curves can be substituted  by taking the Lie bracket with $\vec  h$ on the level of sub-Riemannian structure. The normalization procedure of \cite{icdifferential} can be described in the following steps:


{\bf Step 1 }
The vector field $\mathcal E_a(\lambda)$ can be characterized , uniquely up to a sign, by the following conditions:\  $ \mathcal E_a(\lambda)\in \Pi_\lambda$, $\hbox{ad}{\vec h}\,\mathcal E_a(\lambda)\in \Pi_\lambda$, and
$$\sigma(\hbox{ad}\, {\vec h}\,\mathcal E_a(\lambda), (\hbox {ad}{\vec h})^2\mathcal E_a(\lambda))=1. $$
Then by the first two lines of \eqref{structeq1}  $\mathcal E_b(\lambda)=\hbox{ad}{\vec h}\,\mathcal E_a(\lambda)$ and \ $ \mathcal F_b(\lambda)=(\hbox {ad} {\vec h})^2\mathcal E_a(\lambda)$.

{\bf Step 2 } The subspace $\mathcal V_c$ is uniquely characterized by the following two conditions:
\begin{enumerate}\label{EC}
\item $\mathcal V_c(\lambda)$ is the complement of $\mathcal V_a(\lambda)\oplus \mathcal V_b(\lambda)$ in $\Pi_\lambda$;
\item $\mathcal V_c(\lambda)$ lies in the skew symmetric complement of $$\mathcal V_a(\lambda)\oplus \mathcal V_b(\lambda)\oplus \mathbb R (\hbox {ad} {\vec h})^2\mathcal E_a(\lambda)\oplus \mathbb R (\hbox {ad} {\vec h})^3\mathcal E_a(\lambda).$$
\end{enumerate}
It is endowed with the canonical Euclidean structure, which is the restriction of $\dot {\mathfrak J}_\lambda(0)$ on it.

{\bf Step 3}
The restriction of the parallel transport $\mathcal K^t$ to $\mathcal V_c(\lambda)$ is characterized by the following two properties:

\begin{enumerate}
\item $\mathcal K^t$ is an orthogonal transformation of spaces $\mathcal V_c(\lambda)$ and $\mathcal V_c\bigl(e^{t\vec h}\lambda\bigr)$;
\item The space ${\rm span} \{\frac{d}{dt}\bigl((e^{-t\vec h})_*(\mathcal K^t v)\bigr)|_{_{t=0}}: v\in \mathcal V_c(\lambda)\}$ is isotropic.
\end{enumerate}

%
Then $\mathcal V^{trans}_c(\lambda)= {\rm span} \{\frac{d}{dt}\bigl((e^{-t\vec h})_*(\mathcal K^t v)\bigr)|_{_{t=0}}: v\in \mathcal V_c(\lambda)\}$.

{\bf Step 4} To complete the construction of normal moving frame it remains to fix $\mathcal F_a(\lambda)$. The field $\mathcal F_a(\lambda)$ is uniquely characterized
by the following two conditions (see line 4 of \eqref{structeq1}):
\begin{enumerate}
 \item The tuple $\{\mathcal E_a(\lambda), \mathcal E_b(\lambda), \mathcal E_c(\lambda),\mathcal F_a(\lambda), \mathcal F_b(\lambda),\mathcal  F_c(\lambda)\}$ constitutes a Darboux frame;
\item $\sigma(\hbox {ad}\,\vec h \mathcal F_a(\lambda),\mathcal F_b(\lambda))=0.$
\end{enumerate}
In order to find $\mathcal F_a(\lambda)$, one can  choose any $\widetilde{\mathcal F}_a(\lambda)$ such that $\{\mathcal E_a(\lambda),\mathcal E_b(\lambda), \mathcal E_c(\lambda), \widetilde {\mathcal F}_a(\lambda), \mathcal F_b(\lambda), \mathcal F_c(\lambda)\}$ constitutes a Darboux frame. Then
\begin{equation}\label{F_a}
\mathcal F_a(\lambda)=\widetilde {\mathcal F}_a(\lambda)-\sigma(\hbox{ad}{\vec h}\,\widetilde {\mathcal F}_a(\lambda), \mathcal F_b(\lambda))\mathcal E_a(\lambda).
\end{equation}

\subsection{Preliminary implementation of the algorithm}\label{Imp}
In order to implement the algorithm for the corank 1 sub-Riemannian structure with symmetries, let us analyze the relation between $T^*M$ and $T^*\widetilde M$ in more detail.
The canonical projection $\pi: T^*M\to M$ induces the canonical projection $\tilde \pi:T^*\widetilde M\to \widetilde M$. Let $\Xi$ be the $1$-foliation such that its leaves are integral curves of $\vec u_0$.
Let ${\rm PR}:T^*M\to T^*M/\Xi$ be the canonical projection to the quotient manifold.

Fix a constant $c$. The quotient manifold
$\{u_0=c\}/\Xi$ can be naturally identified with $T^*\widetilde M$.
Indeed, a point $\tilde \lambda$ in $\{u_0=c\}/\Xi$ can be identified with a leaf ${\rm PR}^{-1}(\tilde\lambda)$ of $\Xi$ which has a form $((e^{-tX_0})^*p,e^{tX_0}q)$, where $\lambda=(p,q)\in {\rm PR}^{-1}(\tilde\lambda)$, $q\in M$ and $p\in T_q^* M$. On the other hand, any element in $T^* \widetilde M$ can be identified with a one-parametric family of pairs $(e^{tX_0}q,(e^{-tX_0})^*(p|_\mathcal D))$.
The mapping $I:\{u_0=c\}/\Xi\to T^*\widetilde M $ sending
$(e^{tX_0}q,(e^{-tX_0})^*p)$ to $(e^{tX_0}q,(e^{-tX_0})^*(p|_{\mathcal D}))$ is one-to-one (because $p(X_0)=u_0$ is already prescribed and equal to $c$) and it defines the required identification. Therefore, for any vector field $X$ on $T^*\widetilde M$, we can assign the vector field $\underline{X}$ on $T^*M$ s.t. $PR_*\underline X=(I^{-1})_*X$ and $\pi_*\underline X\in\mathcal{D}$. 


Let $\tilde\sigma$ be the standard symplectic form on $T^*\widetilde M$. Note that $(I\circ\hbox{PR})^*\tilde\sigma$ is a 2-from on $\{u_0=c\}$. Let, as before, $\sigma$ be the standard symplectic form on $T^*M$. Let $\omega_0$ be the 1-form as in subsection \ref{corank1}. Then $\sigma$ and $\pi^*d\omega_0$ induce two 2-forms on $\{u_0=c\}$ by restriction. The following lemma describes the relation between these 2-forms. 

\begin{lemma}\label{decomp1}
The following formula holds on $\{u_0=c\}$.
\begin{equation}
\label{decompeq}
\sigma=(I\circ\hbox{PR})^*\tilde\sigma-u_0\pi^* d\omega_0.
\end{equation}
\end{lemma}
\begin{proof}
First define a $1$-form $\varsigma_0$ on $T^*M$ by
$$\varsigma_0(v)=u_0\omega_0(\pi_*v),\ v\in T_\lambda^*M,\ \lambda=(p,q)\in T^*M,q\in M,p\in T^*_qM.$$
Let $\varsigma$ and $\tilde\varsigma$ be the tautological (Liouville) $1-$forms on $T^*M$ and $T^*\widetilde M$ respectively. Then on the set $\{u_0=c\}$ one has $\varsigma=(I\circ\hbox{PR})^*\tilde\varsigma+\varsigma_0.$ Therefore, by definition of standard symplectic form on a cotangent bundle, we have
\begin{equation}
\label{decompeq1}
\sigma=(I\circ\hbox{PR})^*\tilde\sigma-d\varsigma_0=(I\circ\hbox{PR})^*\tilde\sigma-du_0\wedge\pi^* \omega_0-u_0\pi^* d\omega_0.
\end{equation}
We complete the proof of the lemma by noticing that $d\varsigma_0=u_0 \pi^*d\omega_0$ on $\{u_0=c\}$.
\end{proof}
Before going further, let us introduce some notations.
Given $ v\in T_\lambda T^*_{q}M$ ($\sim T^*_q M$), where $q=\pi(\lambda)$, we can assign a unique vector $v^h\in T_{{\rm pr}(q)}\widetilde M$ to its equivalence class in $T^*M/\mathcal V_a(\lambda)$ by using the identifications \eqref{ident2} and \eqref{ident3}.
  Conversely, to any $X\in T_{\hbox{pr}(q)}\widetilde M$ one can assign an equivalence class of $T_\lambda(T^*
_qM)/V_a(\lambda)$. Denote by $X^v\in T_\lambda T^*_{q}M$ the unique representative of this equivalence class such that $du_0(X^v)=0$.
\begin{lemma}
\label{sigg}
For any  vectors $X, V \in T_\lambda T^*M$ with $\pi_*V =0$ we have
$\sigma(X,v)=g(\pi_*X, V^h).$
\end{lemma}
\begin{proof}
Let $\lambda=(p,q) \in T^*M,\ p\in T_q^*M,q\in M$ and $\varsigma$ be the tautological (Liouville) 1-form on $T^*M$ as before.
Extend the vector $X$  to a vector field  and $V$ to a vertical vector field in a neighbourhood of $\lambda$.
It follows from the definition of the canonical symplectic form and the verticality of $V$ that

\begin{equation*}
\begin{split}
~&\sigma(X,V) =-d\varsigma(X,V)
=
V(\varsigma(X))+\varsigma([X,V])
=\\
~&V(p\cdot \pi_*X)-p\cdot \pi_*[V,X]
= V\cdot\pi_*X.
\end{split}
\end{equation*}
In the last equality here we use again the identification between  $T_\lambda T^*_{q}M$ and $T^*_q M$. Finally, $V\cdot\pi_*X=g(V^h, \pi_*X)$ by the definition of $V^h$.
\end{proof}

Lemma \ref{decomp1} implies that the sub-Riemannian Hamiltonian vector field can be decomposed into the Riemannian Hamiltonian vector field and another part depending on the tensor $J$.

\begin{lemma}\label{decomp}
The following formula holds.
\begin{equation}
\label{srHam}
 \vec h(\lambda)=\underline{\nabla_{p^h}}-u_0(Jp^h)^v,
\end{equation}
where $\lambda=(p,q)\in T^*M, q\in M, p\in T^*_qM $ and $\nabla_{p^h}$ is the lift of $p^h$ to $T^*\widetilde M$ w.r.t. the Levi-Civita connection.
\end{lemma}

\begin{proof}
Denote by $\tilde h$ the Riemannian Hamiltonian function on $T^*\widetilde M$. 
Since the Hamiltonian vector field $\vec {\tilde h}$ is horizontal w.r.t. the Levi-Civita connection and its projection to $\widetilde M$ is equal to $p^h$, we have $\vec{\tilde h}=\nabla_{p^h}.$ Further, it follows from the definition of $I$  that $(I\circ\hbox{PR})^*\tilde h=h$ and $(I\circ\hbox{PR})_*(\underline{\nabla_{p^h}})=\nabla_{p^h}$. Thus, for any vector $X$ tangent to $\{u_0=c\}$, we have
\begin{eqnarray*}
\sigma(\underline{\nabla_{p^h}}, X)&=&((I\circ\hbox{PR})^*\tilde\sigma-u_0\pi^* d\omega_0)(\underline{\nabla_{p^h}}, X)\\
&=&\tilde\sigma(\nabla_{p^h}, (I\circ\hbox{PR})_*X)-u_0 d\omega_0(p^h, \pi_*X)\\
&=&d\tilde h\big((I\circ\hbox{PR})_*X\big)-u_0 d\omega_0(p^h, \pi_*X)\\
&=&(I\circ\hbox{PR})^*d\tilde h(X)-u_0d\omega_0(p^h, \pi_*X)\\
&=&d\big((I\circ\hbox{PR})^*\tilde h\big)(X)-u_0g(Jp^h, \pi_*X)\\
&=&dh(X)+u_0\sigma((Jp^h)^v, X)
\end{eqnarray*}
It follows that $\vec h(\lambda)$ and $\underline{\nabla_{p^h}}-u_0(Jp^h)^v$ are equal modulo $\mathbb R\vec u_0$, which is the symplectic complement of the tangent space to $\{u_0=c\}$. But $\pi_*\vec h(\lambda),\pi_*\big(\underline{\nabla_{p^h}})\in D_q$ and  $\pi_* \vec u_0=X_0\notin D_q$, which implies \eqref{srHam}.
\end{proof}

Now we give more precise description of normal moving frames following the steps as in subsection \ref{algorithm}.
Assume that  $\mathcal V_a^{trans}(\lambda), \mathcal V_b^{trans}(\lambda), \mathcal V_c^{trans}(\lambda)$ are defined by \eqref{vtrans}.

{\bf Step 1 }First define the vector field $\widetilde{\mathcal E_a}$ on $T^*M$ by
\begin{equation}\label{E_a}
\widetilde{\mathcal E_a}(\lambda)\in\Pi_\lambda,\ \widetilde{\mathcal E_a}(\lambda)\in\mathcal D^\bot,\ du_0(\widetilde{\mathcal E_a}(\lambda))=1.
\end{equation}
For further calculations it is convenient  to denote $\widetilde{\mathcal E_a}$ by $\partial_{u_0}$, because to take the Lie brackets of $\widetilde{\mathcal E_a}$ with $\vec h$ is the same as to make ``the partial derivatives w.r.t. $u_0$'' in the left handside of \eqref{srHam}.
Indeed, by \eqref{srHam}  $\hbox{ad}{\vec h}\ \partial_{u_0}=(Jp^h)^v\in\Pi_\lambda$ and then $\pi_*\bigl((\hbox{ad}{\vec h})^2\ \partial_{u_0}\bigr)=-Jp^h$. Then from Lemma \ref{sigg} it follows immediately that
$$\sigma(\hbox{ad}{\vec h}\ \partial_{u_0}, (\hbox{ad}{\vec h})^2\ \partial_{u_0})=\|Jp^h\|^2.$$
As a direct consequence of the last identity we get
 \begin{prop}\label{regular}
 A point $\lambda=(p, q)\in T^*M$ is a $D-$regular point if and only if $J_qp\neq0.$
\end{prop}
\begin{remark}
Note that if $\mathcal D$ is a contact distribution the operators $J_q$ are non-singular, and all points of $T^*M$ out of the zero section are $D$-regular.
\end{remark}
Further from step 1 of subsection \ref{algorithm}, we have that
\begin{eqnarray}
&~&
\mathcal E_a(\lambda)=\frac{\partial_{u_0}}{\|Jp^h\|},\label{Ea}\\
&~&\mathcal E_b(\lambda)=\hbox{ad}{\vec h}\ \mathcal E_a(\lambda)=\frac{(Jp^h)^v}{\|Jp^h\|}+
\vec h\left(\frac{1}{\|Jp^h\|}\right)\partial_{u_0},\label{Eb}\\
&~&\mathcal F_b(\lambda)=\hbox{ad}{\vec h}\ \mathcal E_b(\lambda)=\frac{1}{\|Jp^h\|}[\vec h,(Jp^h)^v]+2
\vec h\left(\frac{1}{\|Jp^h\|}\right)(Jp^h)^v+
(\vec h)^2\left(\frac{1}{\|Jp^h\|}\right)\partial_{u_0}\label{Fb}.
\end{eqnarray}

By direct computations,
\begin{equation}
\label{horfb}
\pi_*[\vec h,(Jp^h)^v]=-Jp^h.
\end{equation}

{\bf Step 2 }
Let us characterize the space $\mathcal V_c(\lambda)$. For this let $\widetilde\Pi_\lambda=\{v\in\Pi_\lambda: du_0(v)=0\}$ and let $\pi_0:\Pi_\lambda\rightarrow\widetilde \Pi_\lambda$ be the projection from $\Pi_\lambda$ to $\widetilde\Pi_\lambda$ parallel to $\mathcal E_a(\lambda)$. Note that $\pi_0(v)=(v^h)^v$.
Since $\mathcal V_c(\lambda)\in \Pi_\lambda$ and $\mathcal V_c(\lambda)$ lies in the skew symmetric complement of $(\hbox{ad}\vec h)^2\mathcal E_a(\lambda)$, we have, using \eqref {horfb} and Lemma \ref{sigg}, that
\begin{equation}
\label{Vc0}
\mathcal V_c(\lambda)\equiv({\rm span}\{(p^h), (Jp^h)\}^\bot)^v\quad {\rm mod}\,\mathbb
R \mathcal E_a(\lambda).
\end{equation}

Further, let $\widetilde {\mathcal V}_c(\lambda)=\pi_0(\mathcal V_c)$.
Using the condition that $\mathcal V_c(\lambda)$ is in the skew symmetric complement of $(\hbox{ad}\vec h)^3\mathcal E_a(\lambda)$, we have
\begin{equation}\label{V_c}
 \mathcal V_c(\lambda)=\{v+\mathcal A(\lambda,v)\mathcal E_a(\lambda):\ v\in\widetilde {\mathcal V}_c(\lambda)\}.
\end{equation}
where $\mathcal A(\lambda, v)$ is the linear functional on the Whitney sum $T^*M\oplus T^*M$ over $M$, given by
\begin{equation}\label{tilde A}
 \mathcal {A}(\lambda,v)=\sigma(v, \frac{(\hbox{ad}\vec h)^2\ (Jp^h)^v}{\|Jp^h\|}).
\end{equation}



{\bf Step 3 }
Since the normal moving frame is a Darboux frame, the space $\mathcal V_c^{\hbox{trans}}(\lambda)$ lies in the skew symmetric complement of $\mathcal V_b(\lambda)$. Besides,
its image under $\pi_*$ belongs to $\mathcal D\bigl(\pi(\lambda)\bigr)$. Then, using Lemma \ref{sigg} we obtain that
\begin{equation}
\label{pivtrans}
{\rm pr}_*\circ\pi_*\bigl(\mathcal V^{\hbox{trans}}_c(\lambda)\bigr)\equiv {\rm span}\{p^h,Jp^h\}^\bot \quad {\rm mod \,\mathbb R p^h},
\end{equation}
where, as before, ${\rm pr}:M\to \widetilde M$ is the canonical projection.
Recall that $\mathcal V^{\hbox{trans}}_c(\lambda)\in T_\lambda(T^*M)/\mathbb R\vec h(\lambda)$.
As a canonical representative of $\mathcal V^{\hbox{trans}}_c(\lambda)$ in $T_\lambda(T^*M)$ one can take the representative, which projects exactly to ${\rm span}\{p^h,Jp^h\}^\bot$ by $\pi_*$. In the sequel, this canonical representative will be denoted by
$\mathcal V^{\hbox{trans}}_c(\lambda)$ as well.

Further, given any $X\in {\rm span}\{p^h,Jp^h\}^\bot$ denote by $\nabla^c_{X}$ the lift of $X$ to $\mathcal V_c^{\hbox{trans}}(\lambda)$: i.e. the unique vector  $\nabla^c_{X}\in V_c^{\hbox{trans}}(\lambda)$ such that ${\rm pr}_*\circ \pi_*\nabla^c_{X}=X$. Then there exist the unique $B\in {\rm End}\bigl(\widetilde {\mathcal V}_c(\lambda)\bigr)$ and $\alpha, \beta\in \mathcal V_c(\lambda)^*$ such that
\begin{equation}
\label{decompc}
\nabla^c_{v^h}=\underline{\nabla_{v^h}}+B\bigl(\pi_0(v)\bigr)+\alpha(v)\frac{(Jp^h)^v}{\|Jp^h\|^2}+\beta(v)\partial_{u_0},\quad \forall v\in \mathcal V_c
\end{equation}
where, as before, $\nabla$ stands for the lifts to the Levi-Civita connection on $T^*\widetilde M$.
Let us describe the operator $B$ and the functionals $\alpha$ and $\beta$ more precisely.
First we prove the following lemma, using the property (1) of the parallel transport $\mathcal K^t$ listed in subsection \ref{algorithm}:
\begin{lemma}\label{anti}
The linear operator $B$ is antisymmetric w.r.t. the canonical
Euclidean structure in $\mathcal V_c(\lambda)$.
\end{lemma}

\begin{proof}
Fix a point $\bar\lambda\in T^*M$ and consider a small neighborhood $U$ of $\bar\lambda$.
Let $\mathcal E_c=\{\mathcal E^i_c\}_{i=1}^{n-3}$ be a frame of $\mathcal V_c)$ (i.e.
$\mathcal V_c(\lambda)={\rm span}\, \mathcal E_c(\lambda)$) for any $\lambda\in U$ such that the following three conditions hold
\begin{enumerate}
\item $\mathcal E_c$ is orthogonal w.r.t. the canonical Euclidean structure on $\mathcal V_c$;
\item Each vector field $\mathcal E_c^i$ is parallel w.r.t the canonical parallel transport $\mathcal K_t$, i.e. $\mathcal E_c^i(e^t\vec h\lambda)=\mathcal K^t\mathcal E_c^i(\lambda)$ for any $\lambda$ and $t$ such that $\lambda, e^{t\vec h}\lambda\in U$;
\item The vector fields $(Jp^h)^v$ and $\mathcal E_c^i$ commute on $U\cap T_{\pi(\bar\lambda)}^*M$;
\item The vector fields $\vec u_0$ and $\mathcal E_c^i$ commute on $U\cap T_{\pi(\bar\lambda)}^*M$.
\end{enumerate}
Note that the frame $\mathcal E_c$ with properties above exists, because the Hamiltonian vector field $\vec h$ is transversal to the fibers of $T^*M$ and it commutes with $\vec u_0$.

From the property (2) of the parallel transport $\mathcal K^t$ (see property (2) in step 3 of subsection \ref{algorithm}) it follows that \begin{equation}\label{trans2}
\nabla^c_{(\mathcal E_c^i)^h}=-{\rm ad}\vec h\, \mathcal E_c^i
\end{equation}

Let $\widetilde {\mathcal E}^i=
\pi_0(\mathcal E_c^i)$ for $1\leq i\leq n-3$ and $\widetilde {\mathcal E}^{n-2}=\frac{(Jp^h)^v}{\|Jp^h\|}$.
Also let $\widetilde {\mathcal E}=\{\widetilde{\mathcal E}^i\}_{i=1}^{n-2}$. Using the above defined identification $I:\{u_0=c\}/\Xi\to T^*\widetilde M$,
where $c=u_0(\bar\lambda)$, one can look on the restriction of the tuple of vector fields $\widetilde E$ to the submanifold $\{u_0=c\}$ as on the tuple of the vertical vector fields of $T^*\widetilde M$ (which actually span the tangent to the intersection of the fiber of $T^*\widetilde M$ with the level to the corresponding Riemannian Hamiltonian). Then first the tuple $\widetilde{\mathcal E}$ is the tuple of
orthonormal vector fields (w.r.t. the canonical Euclidean structure on the fibers of $T^*\widetilde M$, induced by the Riemannian metric $g$). Further, by Remark \ref{lcrem} the Levi-Civita connection of $g$ is characterized by the fact that there exists a field of  antisymmetric operators $\widetilde B\in {\rm End}\bigl({\rm span}\, \widetilde {\mathcal E}(\lambda)\bigr)$ such that
\begin{equation}
\label{nablA}
[\nabla_{p^h},\widetilde{\mathcal E}^i(\lambda)]=-\nabla_{\bigl(\widetilde{\mathcal E}^i(\lambda)\bigr)^h} -\widetilde B \widetilde{\mathcal E}^i(\lambda)
\end{equation}
From \eqref{trans2} and \eqref{nablA}, using \eqref{srHam},\eqref{V_c}, and the property (3) of $\mathcal E_c^i$, one has
\begin{equation}
\label{seria1}
\begin{split}
\nabla^c_{(\mathcal E_c^i)^h}=-{\rm ad}\vec h\, \mathcal E_c^i=-\bigl[\underline{\nabla_{p^h}}-u_0(Jp^h)^v,\widetilde {\mathcal E}^i+\mathcal A(\lambda, \mathcal E^i)\frac{\partial_{u_0}}{\|Jp^h\|}\bigr]\\
=\underline{\nabla_{\bigl(\widetilde{\mathcal E}^i(\lambda)\bigr)^h}} +\widetilde B\,\widetilde{\mathcal E}^i(\lambda)- \mathcal A(\lambda,\mathcal E^i)\frac{(Jp^h)^v}{\|Jp^h\|}\quad {\rm mod}\ \mathbb R\partial_{u_0}.
\end{split}
\end{equation}
Note that one has the following orthogonal splitting of the space ${\rm span} \,\widetilde{\mathcal E}$:
\begin{equation}
\label{ortsplit}
{\rm span} \,\widetilde{\mathcal E}(\lambda)=\widetilde {\mathcal V}_c(\lambda)\oplus \mathbb R(Jp^h)^v.
\end{equation}
The operator  $B$ is exactly the endomorphism of  $\widetilde {\mathcal V}_c(\lambda)\bigr)$ such that $B \tilde v$ is the projection of
$\widetilde B \tilde v$ to  $\widetilde V_c(\lambda)$ w.r.t. the splitting \eqref{ortsplit} for any $\tilde v\in \tilde V^c$.  Obviously, the antisymmetricity of $\widetilde B$ implies the antisymmetricity of $B$. The proof of the lemma is completed.
\end{proof}

Now we are ready to find $B$ explicitly using the fact that $\mathcal V_c^{{\rm trans}}$ is isotropic. For this let $\varphi$ be the projection from $(\mathbb R p^h)^\perp$ to ${\rm span}\{p^h, J p^h\}^\perp$ parallel to $J p^h$. Obviously,
\begin{equation}
\label{vfproj}
\varphi(\tilde v)=\tilde v-g(\tilde v,J p^h)\frac{J p^h}{\|J p^h\|^2},\quad \forall \tilde v\in \widetilde {\mathcal V}_c.
\end{equation}
\begin{lemma}\label{liftB}
The operator $B$ satisfies
\begin{equation}
\label{B1}
(B\tilde v)^h=-\frac{u_0}{2}\varphi \circ J \tilde v^h, \quad \forall \tilde v\in \widetilde {\mathcal V}_c
\end{equation}
or, equivalently,
\begin{equation}
\label{B2}
B \tilde v=\frac{u_0}{2}\left(-(J\tilde v^h)^v+g(J\tilde v^h,Jp^h)\frac{(Jp^h)^v}{\|Jp^h\|^2}\right), \quad \forall \tilde v\in \widetilde {\mathcal V}_c.
\end{equation}
\end{lemma}
\begin{proof}
Since $\mathcal V_c^{{\rm trans}}(\lambda)$ is an isotropic subspace, we have
\begin{equation*}
\sigma(\nabla^c_{v_1^h}, \nabla^c_{v_2^h})=0,\quad  \forall\, v_1,v_2\in \mathcal V_c
\end{equation*}

On the other hand, from \eqref{decompc} and the fact that $V_c^{\rm trans}$ lies in the skew symmetric complement of $V_a\oplus V_b$ it follows that
\begin{equation}
\label{isot2}
\sigma(\nabla^c_{v_1^h}, \nabla^c_{v_2^h})=\sigma\Bigl(\underline{\nabla_{v^h_1}}+B\tilde v_1, \underline{\nabla_{v^h_2}}+B\tilde v_2\Bigr),
\end{equation}
where $\tilde v_i=\pi_0(v_i)$, $i=1,2$.
Then, using \eqref{decompeq}, the fact that the Levi-Civita connection (as an Ehresmann connection) is a Lagrangian distribution in $T^*\widetilde M$ and Lemma \ref{sigg}, we get
\begin{equation*}
\begin{split}
~&0=\sigma(\nabla^c_{v_1^h}, \nabla^c_{v_2^h})=
\Bigl((I\circ\hbox{PR})^*\tilde\sigma-u_0\pi^* d\omega_0\Bigr)\Bigl(\underline{\nabla_{v^h_1}}+B\tilde v_1,
\underline{\nabla_{v^h_2}}+B\tilde v_2\Bigr)=\\
~&-u_0d\omega_0(v_1^h,v_2^h)-g\big((B \tilde v_1)^h, v_2^h\big)+g\big((B \tilde v_2)^h, v_1^h)=\\
~&-u_0g(Jv_1^h, v_2^h)-g\big((B \tilde v_1)^h, v_2^h\big)+g\big((B^*\tilde v_1)^h, v_2^h).
\end{split}
\end{equation*}
Taking into account that $B$ is antisymmetric,  we get identity \eqref{B1}. Then, using relation \eqref{vfproj} and Lemma \ref{sigg}, one easily gets identity \eqref{B2}.
\end{proof}

 Further we need the following notation. Given a map $ S:T^*M\oplus W_\lambda\longrightarrow\mathbb R$, define a map $S^{(1)}:T^*M\oplus T^*M\longrightarrow \mathbb R$ by
\begin{equation}\label{operator}
S^{(1)}(\lambda, v)=\frac{d}{dt}S(e^{t\vec h}\lambda,{\mathcal K}^t v)\left.{\!\!\frac{}{}}\right|_{t=0},\ \lambda , v \in T^*M,
\end{equation}
where in the second argument we use again the natural identification of $T_{\pi(\lambda)}^*M$ with $T_\lambda(T^*_{\pi(\lambda)}M)$.

\begin{lemma}
\label{ablem}
The functionals $\alpha$ and $\beta$ from \eqref{decompc} satisfy the following identities
\vskip .1in
\begin{enumerate}
 \item $\alpha(v)=-\sigma(\underline{\nabla_{v^h}}, \hbox{ad}\vec h\ (Jp^h)^v)$;
 \vskip .1in
 \item $\beta(v)=-\Bigl(\frac{1}{\|Jp^h\|}\mathcal A\Bigr)^{(1)}\bigl(\lambda, (v^h)^v\bigr)=-\frac{1}{\|Jp^h\|}\mathcal{A}^{(1)}\bigl(\lambda,(v^h)^v\bigr)-
 {\vec h}\left(\frac{1}{\|Jp^h\|}\right)\mathcal{A}\bigl(\lambda,(v^h)^v\bigr)$.
\end{enumerate}
\end{lemma}
\begin{proof}

First, from step 2 in subsection \ref{algorithm} it follows  that for any $v\in\mathcal V_c(\lambda)$, we have
\begin{eqnarray*}
 0=\sigma(\nabla^c_{v^h}, \hbox{ad}\vec h\ (Jp^h)^v)
&=&\sigma(\underline{\nabla_{v^h}}+B\bigl(\pi_0(v)\bigr)+\alpha(v)\frac{(Jp^h)^v}{\|Jp^h\|^2}+\beta(v)\partial_{u_0}, \hbox{ad}\vec h\ (Jp^h)^v)=\\
&~&\sigma(\underline{\nabla_{v^h}}, \hbox{ad}\vec h\ (Jp^h)^v)+\alpha(v).
\end{eqnarray*}
Therefore, $\alpha(v)=-\sigma(\underline{\nabla_{v^h}}, \hbox{ad}\vec h\ (Jp^h)^v)$.

Further, take the tuple of vertical vector fields $\mathcal E_c=\{\mathcal E_c^i\}_{i=1}^{n-3}$ as in the proof of Lemma \ref{anti}. Then from \eqref{decompc},\eqref{trans2}, and the fact that the vector fields $\vec h$ and $\vec u_o$ commute it follows that
\begin{equation}
\label{seria2}
\beta(\mathcal E_c^i)=\sigma(\vec u_0, \nabla^c_{(\mathcal E_c^i)^h})=-\sigma(\vec u_0, {\rm ad} \,\vec h\mathcal E_c^i)=
-[\vec h,\mathcal E_c^i](u_0)=-\vec h \circ \mathcal E_c^i(u_0)=-\vec h\bigl(\sigma(\vec u_0,\mathcal E_c^i)\bigr).
\end{equation}
Then from by \eqref{V_c} it follows
\begin{equation}
\label{seria3}
\sigma(\vec u_0,\mathcal E_c^i)=\frac{1}{\|J p^h\|}\mathcal A(\lambda, \widetilde{\mathcal E}_c^i).
\end{equation}
The item (2) of the lemma follows immediately from \eqref{seria2} and \eqref{seria3}.
\end{proof}



{\bf Step 4 }
 According to the algorithm, described in subsection \ref{algorithm}, first find some vector field $\widetilde{\mathcal F}_a$ such that the tuple $\{\mathcal E_a, \mathcal E_b, \mathcal E_c, \widetilde{\mathcal F}_a, \mathcal F_b, \mathcal F_c\}$ constitutes a Darboux frame. Let $\mathfrak V_0$ be a vector in ${\mathcal V}_c(\lambda)$ such that
 \begin{equation}
 \label{mathfrV0}
 \sigma(\mathfrak V_0,\nabla_{v^h}^c)=\beta(v),\quad \forall v\in\mathcal V_c(\lambda).
 \end{equation}
 Also, let $\mathfrak W_0$ be a vector in $\mathcal V^{\hbox{trans}}_c(\lambda)$ such that
 \begin{equation}
 \label{mathfrW0}
 \sigma(v,\mathfrak W_0)=\mathcal A(\lambda,v), \quad \forall v\in\mathcal V_c(\lambda).
 \end{equation}
 Note that by constructions the map $v\mapsto \nabla_{v^h}^c$ is an isomorphism between $\mathcal V_c$ and $\mathcal V^{\rm trans}_c$. Let $\mathfrak V_1$ be a vector in $\mathcal V_c$ such that $\mathfrak W_0=\nabla_{\mathfrak V_1^h}^c$. Then from \eqref{mathfrV0} and \eqref{mathfrW0} it follows that
\begin{equation}
\label{Abeta}
\mathcal A(\lambda, \mathfrak V_0)=\beta(\mathfrak V_1).
\end{equation}

 \begin{lemma}
 A vector field $\widetilde{\mathcal F}_a$ can be taken in the following form
 \begin{equation}\label{Fa}
\widetilde{\mathcal F}_a(\lambda)=-\|Jp^h\|\vec
 u_0+\|Jp^h\|\mathfrak V_0
-\mathfrak W_0+\|Jp^h\|(\vec h)^2\left(\frac{1}{\|Jp^h\|}\right)\mathcal E_b(\lambda)-\|Jp^h\|\vec h\left(\frac{1}{\|Jp^h\|}\right)\mathcal F_b(\lambda)
\end{equation}
 \end{lemma}
 \begin{proof}
 Note that such vector field $\widetilde{\mathcal F}_a$
 is defined modulo $\mathbb R \mathcal E_a=\mathbb R \partial_{u_0}$. Therefore we can look for $\widetilde{\mathcal F}_a$ in the form
 \begin{equation}
 \label{Faprelim}
 \widetilde{\mathcal F}_a=\gamma_1\vec u_0+\gamma_2\mathcal E_b+\gamma_3\mathcal F_b+v_c+\bar v_c,
 \end{equation}
 where $v_c\in\mathcal V_c$ and $\bar v_c\in\mathcal V^{{\rm trans}}_c$.
 Then
 \begin{enumerate}
 \item From relations $\sigma (\mathcal E_a, \widetilde{\mathcal F}_a)=1$ and \eqref{Ea} it follows that $\gamma_1=-\|Jp^h\|$;
 \item From relations $\sigma (\mathcal E_b, \widetilde{\mathcal F}_a)=0$ and \eqref{Eb} it follows that $\gamma_3=-\|Jp^h\|\vec h\left(\frac{1}{\|Jp^h\|}\right)$;
  \item From relations $\sigma (\mathcal F_b, \widetilde{\mathcal F}_a)=0$ and \eqref{Fb} it follows that $\gamma_2=\|Jp^h\|(\vec h)^2\left(\frac{1}{\|Jp^h\|}\right)$;
 \item From relations $\sigma(\widetilde{\mathcal F}_a,\nabla_v^c)=0$ for any $v\in \mathcal V_c$ and the decomposition
 \eqref{decompc} it follows that $\sigma(v_c,\nabla_{v^h}^c)=\|Jp^h\|\beta(v)$ for any $v\in \mathcal V_c$. Hence $v_c=\|J p^h\|\mathfrak V_0$;
 \item From relations $\sigma(\widetilde{\mathcal F}_a, v)=0$ for any $v\in \mathcal V_c$ and relation \eqref{V_c} it follows that $\sigma (\bar v_c, v)=\mathcal A(\lambda,v)$ for any $v\in \mathcal V_c$. Hence $\bar v_c=-\mathfrak W_0$.
 \end{enumerate}
 Combining items (1)-(5) above we get \eqref{Faprelim}.
 \end{proof}
 The canonical $\mathcal F_a$ is obtained from $\widetilde{\mathcal F}_a$ by formula \eqref{F_a}.

\vskip .2in
Now as a direct consequence of structure equation \eqref{structeq1}, we get the following preliminary descriptions of $(a, b)-$ curvature maps (under identification \ref{ident4}).
\begin{prop}\label{preli}
Let $V$ be a parallel vector field such that $V(\lambda)=v$.
Then the curvature maps satisfy the following identities:
\begin{eqnarray}~&g\big((\mathfrak R_\lambda(c,c)v)^h,w^h\big) =-\sigma(\hbox{ad}\vec h\ \nabla^c_{V^h},\nabla^c_{w^h}),\quad \forall w\in \mathcal V_c(\lambda)\label{R1}\\
~&\mathfrak R_\lambda(c,b)v=\sigma(\hbox{ad}\vec h\ \nabla^c_{V^h},\mathcal F_b(\lambda))\frac{(Jp^h)^v}{\|Jp^h\|}=
\sigma( \hbox{ad}\vec h\ \mathcal F_b(\lambda),\nabla^c_{v^h})\frac{(Jp^h)^v}{\|Jp^h\|}
\label{R2}\\
~&\mathfrak R_\lambda(c,a)v=\sigma(\hbox{ad}\vec h\ \nabla^c_{V^h},\mathcal F_a(\lambda))\partial_{u_0}\label{R3}\\
~&\mathfrak R_\lambda(b,b)(\frac{(Jp^h)^v}{\|Jp^h\|})=-\sigma(\hbox{ad}\vec h\ \mathcal F_b(\lambda),\mathcal F_b(\lambda))(\frac{(Jp^h)^v}{\|Jp^h\|})\label{R4}\\
~&\mathfrak R_\lambda(a,a)\partial_{u_0}=-\sigma(\hbox{ad}\vec h\ \mathcal F_a(\lambda),\mathcal F_a(\lambda))\partial_{u_0}\label{R5}
\end{eqnarray}
\end{prop}


\section{Calculus and the canonical splitting}
\setcounter{equation}{0}
\setcounter{theor}{0}
\setcounter{lemma}{0}
\setcounter{prop}{0}
\subsection{Some useful formulas}
Constructions of the previous section
 show that in order to calculate the $(a, b)-$ curvature maps it is sufficient to know how to express the Lie bracket of vector fields on the cotangent bundle $T^*M$ via the covariant derivatives of Levi-Civita connection on $T^*\widetilde M$. For this, we need special calculus which will be given in Proposition \ref{basic} below.
Let $A$ be a tensor of type $(1,K)$ and $B$ be a tensor of type $(1,
N)$ on $\widetilde M$, $K, N\geq 0$. Define a new tensor $A\bullet B$ of type
$(1, K+N-1)$ by
$$A\bullet B(X_1,...,X_{K+N-1})=\sum_{i=0}^{K-1}A(X_1,...,X_i,B(X_{i+1},\ldots X_{i+N}),X_{i+N+1},...,X_{K+N-1}).$$
 This definition needs a clarification in the cases when either $K=0$ or $N=0$. If $K=0$, then we set $A\bullet B=0$, and if $N=0$, i.e. $B$ is a vector field on $\widetilde M$, then we set $A\bullet B (X_1,...,X_{K-1})=\sum_{i=0}^{K-1}A(X_1,...,X_i,B,X_{i+1},...,X_{K-1}).$
Also define by induction $A^{i+1}=A\bullet A^{i}$.
For simplicity,
in this section, we denote
\begin{equation}\label{simple}
Ap^h=A(\underbrace{p^h,p^h,...,p^h}_{K}),\ Ap=(Ap^h)^v.
\end{equation}
Besides, we denote by $\nabla A$ the covariant derivative (w.r.t.
the Levi-Civita connection) of the tensor $A$, i.e., $\nabla A$ is a
tensor of type $(1, K+1)$ defined by
\begin{equation}\label{covariant}
\nabla A(X_1, ..., X_K, X_{K+1})=(\nabla_{X_{K+1}} A)(X_1, ...,
X_K).
\end{equation}
Also define by induction $\nabla^{i+1}A=\nabla(\nabla^i A).$

\medskip
 Now we are ready to give several formulas, relating  Lie derivatives w.r.t. the $\vec h$ and classical covariant derivatives, which will be the base for our further calculations:
\begin{prop}\label{basic}
The following identities hold:
\begin{enumerate}
\item $[Ap, Bp]=(B\bullet A)p-(A\bullet B)p;$
\item $[\underline{\nabla_{Ap^h}}, Bp]=-\underline{\nabla_{(A\bullet B)p^h}}+((\nabla_{Ap^h} B)p^h)^v;$
\vspace{2mm}
\item $[\underline{\nabla_{Ap^h}},\underline{\nabla_{Bp^h}}]=
\underline{\nabla_{
(\nabla_{Ap^h}B)p^h-
(\nabla_{Bp^h}A)p^h}
}+(R^\nabla(Ap^h,
Bp^h)p^h)^v-
\Omega(Ap^h,Bp^h)\vec u_0,$\\
where the $2$-form $\Omega$ is as in subsection \ref{corank1} (recall that $\Omega(X,Y)=g(JX,Y)$).
\vskip .1in
\item $\underline{\nabla_{p}}\big(g(Ap^h, Bp^h)\big)=g\big((\nabla A)p^h, B p^h\big)+ g\big(Ap^h, (\nabla B)p^h\big).$
\end{enumerate}
\end{prop}
\begin{proof}
Obviously, it is sufficient to prove all items of the proposition in the case, when the tensors $A$ and $B$ have the form $A=SX$ and  $B=TY$, where $S$ and $T$ are tensors of the type $(0, K)$ and $(0,L)$ respectively and $X$ and $Y$ are vector fields.
By analogy with \eqref{simple}, let $$S p^h=S(\underbrace{p^h,p^h,...,p^h}_{K}) \text{ and  } T p^h=T(\underbrace{p^h,p^h,...,p^h}_{L}).$$
Then directly from definitions we have
\begin{equation}
\label{basica}
(A\bullet B) p^h= Bp(Sp^h)X,
\end{equation}
where by $Bp(Sp^h)$ we mean the derivative of the function $Sp^h$ in the direction $Bp$.
Therefore
$$[Ap, Bp]=[Sp^hX^v,Tp^hY^v]=Ap(Tp^h)Y^v-Bp(Sp^h)X^v=(B\bullet A)p-(A\bullet B)p,$$
which completes the proof of item (1).

For the proof of the remaining items one can use the following scheme: First one shows that it is sufficient
to prove  them in the case $K=L=0$, i.e. when $A$ and $B$ are vector fields in $\widetilde M$. Then one checks them in the latter case. As a matter of fact, the required identities in the latter case follow directly from the definitions of the Levi-Civita connection  for items (2) and (4) and from the definition of the Riemannian curvature tensor for item (3), where the nonholonomicity of the distribution $\mathcal D$ causes the appearance of the additional term.

Let us prove item (2). The left handside of the required identity for $A=SX$ and  $B=TY$ has the form
\begin{equation}
\label{left2}
[\underline{\nabla_{Ap^h}}, Bp]=[\underline{\nabla_{S p^hX}}, Tp^hY^v]=S p^h X(Tp^h) Y^v-B p(S p^h)\underline{\nabla_X}+Sp^h Tp^h[\underline{\nabla_X}, Y^v]
\end{equation}
Using \eqref{basica}, the first term in the right handside of the required identity can be written as follows:
\begin{equation}
\label{right21}
\underline{\nabla_{(A\bullet B)p^h}}=B p(S p^h)\underline{\nabla_X}.
\end{equation}
Further, let us analyze the second term of the right handside of the required identity:
\begin{equation}
\label{right22}
(\nabla_{Ap^h} B)p^h=(\nabla_{S p^h X} T p^h Y)p^h= S p^h X(Tp^h) Y+Sp^h Tp^h \nabla_X Y
\end{equation}

Comparing \eqref{left2} with \eqref{right21} and \eqref{right22} we conclude that in order to prove the item (2) it is sufficient to show that $[\underline{\nabla_{X}}, (Y)^v]=(\nabla_{X} Y)^v$. The last identity directly follows from the definition of the covariant derivative.

Let us prove item (3). The required identity is equivalent to the following one
\begin{equation}
\label{basic3}
[\underline{\nabla_{Ap^h}},\underline{\nabla_{Bp^h}}]-\underline{\nabla_{
(\nabla_{Ap^h}B)p^h-
(\nabla_{Bp^h}A)p^h}
}
=(R^\nabla(Ap^h,
Bp^h)p^h)^v-\Omega(Ap^h,Bp^h)\vec u_0.
\end{equation}
Note that both sides of the last identity are tensorial: the result of the substitution $A=SX$ to both of them is equal to
$S$ multiplied by the result of the substitution of $A=X$ (and the same for the corresponding substitutions of $B$). Therefore it is sufficient to prove this identity in the case when $A=X$ and $B=Y$, where $X$ and $Y$ are vector fields on $\widetilde M$. Since the Levi-Civita connection is torsion-free, i.e. $\nabla_X Y-\nabla_Y X=[X,Y]$, the required identity in this case has the form
\begin{equation}
\label{basic31}
\left([\underline{\nabla_{X}},\underline{\nabla_{Y}}]-
\underline{\nabla_{[X,Y]
}}\right)(\lambda)=(R^\nabla(X,
Y)p^h)^v-\Omega(X,Y)\vec u_0(\lambda).
\end{equation}
Let us prove identity \eqref{basic31}.
For this let $
\mathcal D^L=\{v\in T_\lambda T^*M:\ \pi_*v\in\mathcal D_q\}
$ be the pullback of the distribution $D$ w.r.t. the canonical projection $\pi$.
Then we have the following splitting of the tangent space $T_\lambda T^*M$ to the cotangent bundle at any point $\lambda$:
\begin{equation}
\label{basic32}
T_\lambda T^*M=\mathcal D^L(\lambda)\oplus\mathbb R\vec u_0.
\end{equation}
Denote by $\pi^L_1$ and $\pi^L_2$ the projection onto $\mathcal D^L$ and the projection onto $\mathbb R\vec u_0$ w.r.t. the splitting \eqref{basic32}, respectively.
By definition, for any vector field $Z$ on $\widetilde M$, one has $\underline{\nabla_{
Z}}\in\mathcal D^L$. Thus by definition of the Riemannian curvature tensor,
\begin{equation}
\label{basic34}
(R^\nabla(X,
Y)p^h)^v=
\pi_1^L\bigl([\underline{\nabla_{X}},\underline{\nabla_{Y}}](\lambda)\bigr)-
\underline{\nabla_{[X,Y]}}(\lambda)
\end{equation}
It remains only to prove that
\begin{equation}
\label{basic35}
\pi^L_2\bigl([\underline{\nabla_{X}},\underline{\nabla_{Y}}]\bigr)=
-
\Omega(X,Y)\vec u_0.
\end{equation}
 Note that from \eqref{decompeq1} it follows that $\mathcal D^L$ is the symplectic complement of the vector field $\partial_{u_0}.$
 Besides, by definition, $\sigma(\vec u_0,\partial u_0)=1$. Therefore,
 \begin{equation}
 \label{basic36}
 \pi^L_2\bigl([\underline{\nabla_{X}},\underline{\nabla_{Y}}]\bigr)=\sigma([\underline{\nabla_{X}}, \underline{\nabla_{Y}}],\partial_{u_0})\vec u_0
 \end{equation}
 Using again \eqref{decompeq1} and the definition of the form $\Omega$ we get
 $$ \sigma([\underline{\nabla_{X}}, \underline{\nabla_{Y}}],\partial_{u_0})=\omega_0(\pi_*[\underline{\nabla_{X}}, \underline{\nabla_{Y}}])=-d\omega_0(\pi_*\underline{\nabla_{X}}, \pi_*\underline{\nabla_{Y}})=-\Omega(X,Y),$$
 where $\omega_0$ is the $1$-form on $M$ defined in subsection \ref{corank1}.
 This completes the proof of the formula \eqref{basic35} and of the item (3).


Finally, let us prove item (4). As in the proof of item (2), we can substitute into the left handside and right handside of the required identity $A=SX$ and $B=TX$ to conclude that  it is sufficient to show that
$$p^h\big(g(X, Y)\big)=g\big(\nabla_{p^h}Y,Y\big)+ g\big(X, \nabla_{p^h}Y),$$
but the latter  is actually the compatibility of the Levi-Civita connection with the Riemannian metric.
\end{proof}


\begin{remark}
Note that if $K=0$ then item (2) has the form
\begin{equation}\label{const1}
[\underline{\nabla_{A}}, Bp]=((\nabla_{A^h} B)p^h)^v
\end{equation}
and if $N=0$ then item (2) has the form
\begin{equation}\label{const2}
[\underline{\nabla_{Ap^h}}, B]=-\underline{\nabla_{(A\bullet B)p^h}};
\end{equation}
\end{remark}

\subsection{Calculation of the canonical splitting}
Using formulas given by Proposition \ref{basic}, we are ready to express the canonical splitting of $W_\lambda$ ($=T_\lambda\mathcal H_{\frac{1}{2}}/\mathbb R\vec h$) in terms of the Riemannian structure and the tensor $J$ on $\widetilde M$.
Note that by \eqref{Ea} the subspace $\mathcal V_a$ is already expressed in this way. To express the subspace
$\mathcal V_b$ and $\mathcal V_b^{\rm{trans}}$ we need the following

\begin{lemma}\label{Jp}
The following identities hold:
\begin{enumerate}
 \item$\vec h\left(\frac{1}{\|Jp^h\|}\right)=-\frac{1}{\|Jp\|^3}g(Jp^h,\nabla J(p^h,p^h));$
 \vspace{1.5mm}
 \item$ (\vec h)^2\left(\frac{1}{\|Jp^h\|}\right)=\frac{3}{\|Jp^h\|^5}g^2(Jp^h,\nabla J(p^h,p^h))-\frac{1}{\|Jp^h\|^3}g(\nabla J(p^h,p^h),\nabla J(p^h, p^h))\\-\frac{1}{\|Jp^h\|^3}g(Jp^h,\nabla^2J(p^h,p^h,p^h))+\frac{u_0}{\|Jp^h\|^3}\big(g(J^2p^h,\nabla J(p^h,p^h))+g(Jp^h,\nabla J(Jp^h,p^h))\\+g(Jp^h,\nabla J(p^h,Jp^h))\big).$
\end{enumerate}

\end{lemma}
\begin{proof}
(1) Using item
(4) of Proposition \ref{basic} we have
\begin{equation}
\label{Jp1}
\nabla_{p^h}\left(g(Jp^h, Jp^h)\right)=2g(\nabla J(p^h,p^h), Jp^h);
\end{equation}
Besides,
\begin{equation}
\label{Jp2}
(Jp^h)^v\left(g(Jp^h, Jp^h)\right)=2g(J^2p^h,Jp^h)=0.
\end{equation}
 Combining the last two identities with \eqref{srHam} we immediately get the first item of the lemma.

(2) Using
item (4) of Proposition \ref{basic},
we get from \eqref{Jp1} that
\begin{eqnarray*}
&~&\nabla_{p^h}^2\left(g(Jp^h, Jp^h)\right)=2\nabla_{p^h}\left(g(\nabla J(p^h,p^h), Jp^h)\right)=2g(\nabla^2 J(p^h,p^h,p^h), Jp^h)\\
&~&+2g(\nabla J(p^h,p^h), \nabla J(p^h,p^h));
\end{eqnarray*}
Further,
\begin{eqnarray*}
&~&(Jp^h)^v\left(g(\nabla J(p^h,p^h), Jp^h)\right)=\\&~&
\left(g(\nabla J(Jp^h,p^h), Jp^h)\right)+\left(g(\nabla J(p^h,Jp^h), Jp^h)\right)+\left(g(\nabla J(p^h,p^h), J^2p^h)\right)
\end{eqnarray*}
Using the last two identities together with \eqref{Jp2}, one can get the second item of the lemma by straightforward computations. 
\end{proof}

Now substituting item (1) of Lemma \ref{Jp} into \eqref{Eb} we get the expression for the subspace $\mathcal V_b$.
Now let us find the expression for $\mathcal V_b^{\rm{trans}}$.
First by \eqref{srHam} and item (2) of Proposition \ref{basic}
we have
\begin{equation}
\label{Fb1}
[\vec h,(Jp^h)^v]=[\underline{\nabla_{p^h}}-u_0(Jp^h)^v,(Jp^h)^v]=-\nabla_{(Jp^h)^v}+(\nabla
J(p^h,p^h))^v
\end{equation}
Substituting the last formula and the items (1) and (2) of Lemma \ref{Jp} into \eqref{Fb} we will get the required expression
for $\mathcal V_b^{\rm trans}$.

Further, according to \eqref{V_c} in order to find the expression for $V_c$ we have to express $\mathcal A(\lambda, v)$.

\begin{lemma} \label{fun1}Let $v\in\Pi_{\lambda}$. Then
\begin{equation}\label{A}
\mathcal A(\lambda, v)=\frac{2}{\|Jp^h\|} g(v^h, \nabla J(p^h,
p^h))-\frac{u_0}{\|Jp^h\|}g(v^h, J^2p^h).
\end{equation}
\end{lemma}
\begin{proof}
Using  relation \eqref{Fb1} and items (2) and (3)  of Proposition \ref{basic}, we get
$$\pi_*\bigl(\hbox{ad}\vec h)^2 (Jp^h)^v\bigr)=-2\nabla
J(p^h,p^h)+u_0J^2p^h.$$ Then
\begin{eqnarray*}
 \sigma(v, \frac{1}{\|Jp^h\|}\hbox{ad}^2\overrightarrow h(Jp^h)^v)&=&\frac{1}{\|Jp^h\|} \sigma(v, -2\nabla J(p^h,p^h)+u_0J^2p^h+\|Jp^h\|^2p^h)\\
&=&\frac{2}{\|Jp^h\|} g(v^h, \nabla J(p^h,
p^h))-\frac{u_0}{\|Jp^h\|}g(v^h, J^2p^h),
\end{eqnarray*}
which completes the proof of the lemma.
\end{proof}
%
In order to express $\mathcal V_c^{\rm{trans}}(\lambda)$ it is sufficient to express the operator $B$ and functionals $\alpha$ and
$\beta$, defined by \eqref{decompc}. The operator $B$ is already expressed by \eqref{B2}.
Further, from decomposition \eqref{decompeq}, Lemma \ref{sigg}, and the fact that the Levi-Civita connection is a Lagrangian distribution it follows that
\begin{eqnarray}\label{alpha}
 \alpha(v)&=&-\sigma(\underline{\nabla_{v^h}}, -\underline{\nabla_{Jp^h}}+(\nabla J(p^h,p^h))^v)\\
\notag&=&-u_0d\omega_0(v^h, Jp^h)-g(v^h,\nabla J(p^h,p^h))\\
\notag&=&u_0g(v^h,J^2p^h)-g(v^h, \nabla J(p^h,p^h))
\end{eqnarray}
Note that from  \eqref{B2}, \eqref{A}, and \eqref{alpha} it follows by straightforward computations that
\begin{equation}
\label{Balpha}
B\bigl(\pi_0(v)\bigr)+\alpha(v)\frac{(Jp^h)^v}{\|Jp^h\|^2}=-\frac{u_0}{2}(J v^h)^v-\frac{1}{2}\mathcal A(\lambda,v)\frac{(Jp^h)^v}{\|Jp^h\|}.
\end{equation}

To derive the formula for $\beta$ we need to study the operator
$\mathcal A^{(1)}$. For later use we will work in more general setting.
Let $\mathfrak S$ be a tensor of type $(1, K)$ on $\widetilde M$.
This tensor induces a map
$S:T^*M\oplus T^*M\longrightarrow\mathbb R$ by
\begin{equation}
\label{sdef}
S(\lambda,v)=g(\mathfrak Sp^h,v^h),\ \lambda=(p,q)\in T^*M, p\in M,p\in T^*_qM.
\end{equation}
where $\mathfrak Sp^h$ is as in \eqref{simple}.
\begin{prop}\label{2tensor}
Let $v\in\mathcal V_c(\lambda)$.
\begin{equation*}
 S^{(1)}(\lambda,v)=-\frac{1}{2}S\left(\lambda, \frac{(Jp^h)^v}{\|Jp^h\|}\right)\mathcal A(\lambda,v)+g(v^h,(\nabla \mathfrak S)p^h-u_0(\mathfrak S\bullet J)p^h+\frac{1}{2}u_0(J\bullet\mathfrak  S)p^h)
\end{equation*}
\end{prop}
\begin{proof}
Take $v\in\mathcal V_c(\lambda)$ and
let $\tilde v=\pi_0(v)$.
Let $V$ and $\widetilde V$ be parallel vector fields such that
$V(\lambda)=v$ and $\widetilde V(\lambda)=\widetilde v$.
We first show that the following identity holds.
\begin{equation}\label{transt}
[\vec h,
\widetilde V](\lambda)=-\nabla_{\tilde v^h}-\frac{1}{2}\mathcal A(\lambda, \tilde v)\frac{(Jp^h)^v}{\|Jp^h\|}+\frac{u_0}{2}(J\tilde v^h)^v.
\end{equation}
For this first by \eqref{decompc} and \eqref{trans2} we have
\begin{equation}
\label{transr}
[\vec h, V](\lambda)=-\nabla_{\tilde v^h}-B(\tilde v)-\alpha(v)\frac{(Jp^h)^v}{\|Jp^h\|^2}-\beta( v)\partial_{u_0}.
\end{equation}
On the other hand from \eqref{V_c} it follows that  $v=\tilde v+\mathcal A(\lambda, \tilde v)\mathcal
E_a(\lambda)$.  Hence from \eqref{Ea}, \eqref{Eb},  and the second relation of Lemma \ref{ablem} one gets
\begin{equation*}
\begin{split}
~&[\vec h, V](\lambda)-[\vec h,\widetilde V](\lambda)=[\vec h, \mathcal A(\lambda, \tilde v)\mathcal
E_a(\lambda)]\\~&=\mathcal A(\lambda, \tilde v)\frac{(Jp^h)^v}{\|Jp^h\|}+\Bigl(\frac{1}{\|Jp^h\|}\mathcal A\Bigr)^{(1)}\bigl(\lambda, \tilde v \bigr)\partial_{u_0}=
\mathcal A(\lambda, \tilde v)\frac{(Jp^h)^v}{\|Jp^h\|}-\beta(v)\partial_{u_0}.
\end{split}
\end{equation*}
Therefore, by \eqref{transr} and \eqref{Balpha}
we
have
\begin{eqnarray*}
 \frac{d}{dt}\left.{\!\!\frac{}{}}\right|_{t=0} e^{-t\overrightarrow H}\tilde v(t)&=&-\nabla_{\tilde v^h}-B(\tilde v)-\alpha( v)\frac{(Jp^h)^v}{\|Jp^h\|^2}-\mathcal A(\lambda, \tilde v)\frac{(Jp^h)^v}{\|Jp^h\|}\\
&=&-\nabla_{\tilde v^h}-\frac{1}{2}\mathcal A(\lambda, \tilde v)\frac{(Jp^h)^v}{\|Jp^h\|}+\frac{u_0}{2}(J\tilde v^h)^v
\end{eqnarray*}
The proof of \eqref{transt} is completed.

\medskip
Further, from Lemma \ref{sigg} and definition of $S$  given by \eqref{sdef} it follows that
$$S(\lambda,v)=\sigma (v, \underline{\nabla_{\mathfrak Sp^h}})$$

\begin{eqnarray}\label{S1}
 S^{(1)}(\lambda, v)
=\sigma\left(
[\vec h, \widetilde V(\lambda)],\underline{\nabla_{\mathfrak Sp^h}}\right)+\sigma(\tilde v,[\vec h,\nabla_{\mathfrak Sp^h}])
\end{eqnarray}
The first term in identity \eqref{S1} can be calculated using the relation \eqref{transt} and Lemmas \ref{sigg} and  \ref{decomp}. Then we apply Proposition \ref{basic} and relation \eqref{decomp} to get $\pi_*\left(\hbox{ad}\vec h(\nabla_{\mathfrak Sp^h})\right)=(\nabla \mathfrak S)p^h-u_0(\mathfrak S\bullet J)p^h$ and we can calculate the second term using again Lemma \ref{sigg}. Putting all the calculations together, we completed the proof of the proposition.
\end{proof}

As a straightforward consequence of the previous Proposition and lemma \ref{Jp}  we get

\begin{cor}\label{A^1}
Let $v\in\mathcal V_c(\lambda)$.
\begin{equation}\
\begin{split}
\label{A^1eq}
~&\mathcal A^{(1)}(\lambda, v)=\frac{1}{\|Jp^h\|}g\left(v^h,2\nabla^2J(p^h,p^h,p^h)-3u_0\nabla J(Jp^h,p^h)\right.\\
~&\left.-2u_0\nabla J(p^h,Jp^h)+\frac{1}{2}u_0^2J^3p^h\right)
-\mathcal A(\lambda, v)\mathcal A\left(\lambda,
\frac{(Jp^h)^v}{\|Jp^h\|}\right).
\end{split}
\end{equation}
\end{cor}
The function $\beta$ can be expressed by substituting \eqref{A^1eq} and item (1) of Lemma \ref{Jp} into item (2) of Lemma \ref{ablem}.
In this way one gets the required expression for  the subspace
$\mathcal V_c^{\rm{trans}}(\lambda)$. To summarize, we have
\begin{equation}
\label{nablac}
\nabla^c_{v^h}=\underline{\nabla_{v^h}}-\frac{1}{2}\mathcal A(\lambda,v)\frac{(Jp^h)^v}{\|Jp^h\|}-\frac{u_0}{2}(Jv^h)^v+\beta(v)\partial_{u_0}.
\end{equation}

To finish the representation of the canonical splitting, we find more detailed expression
for $\mathcal V_a^{\rm{trans}}(\lambda)=\mathbb R\mathcal
F_a(\lambda)$ on the base of  equations \eqref {F_a} and \eqref{Fa}. For this we
will describe the properties of vectors $\mathfrak V_0$, $\mathfrak V_1$, and  $\mathfrak W_0$ from Step 4 of subsection \ref{Imp}
which will be used in the calculations of the curvature maps (section 5).
\begin{lemma}\label{V0W0}
Let $v\in\mathcal V_c(\lambda)$ and $V$ be a parallel vector field such that $V(\lambda)=v$. Then the following identities hold:
\begin{enumerate}
\item $\mathfrak V_1^h=({\rm pr}\circ \pi)_*\mathfrak W_0=-\frac{2}{\|Jp^h\|}\nabla J(p^h,p^h)+\frac{u_0}{\|Jp^h\|}J^2p^h+u_0\|Jp^h\|p^h+\frac{2}{\|Jp^h\|^3}g(\nabla J(p^h,p^h),Jp^h)Jp^h.$
\item
$\sigma\Bigr(\mathfrak W_0,\hbox{ad}\vec h (\nabla_{V^h}^c\bigr)\Bigl)=g\bigg((\mathfrak R_\lambda(c,c)v\big)^h,
\mathfrak V_1^h\bigg),$\\
     $\sigma(\mathfrak W_0,\hbox{ad}\vec  h{\mathcal F}_b(\lambda))=-g\bigg(\bigl(\mathfrak R_\lambda(c,b)\mathfrak V_1\bigr)^h
,\frac{Jp^h}{\|Jp^h\|}\bigg)$;

\vspace{1.5mm}

\end{enumerate}
\end{lemma}
\begin{proof}
\indent





(1)
From \eqref{mathfrW0} and Lemma \ref{fun1} it follows that
$$({\rm pr}\circ\pi)_*\mathfrak{W}_0=-\frac{2}{\|Jp^h\|}\nabla
J(p^h,p^h)+\frac{u_0}{\|Jp^h\|}J^2p^h,\, {\rm mod}\, {\rm span} \{p^h,J p^h\}.$$ Note that by constructions $({\rm pr}\circ\pi)_*\mathfrak{W}_0\in {\rm span}\{p^h, Jp^h\}^\perp$.
Let us work with the orthogonal splitting $T_q\widetilde M={\rm span}\{p^h, Jp^h\}^\perp\oplus \mathbb R p^h\oplus \mathbb R Jp^h$.
Assume that the vector $\frac{2}{\|Jp^h\|}\nabla
J(p^h,p^h)-\frac{u_0}{\|Jp^h\|}J^2p^h$
 has the following decomposition w.r.t. this splitting:
$$\frac{2}{\|Jp^h\|}\nabla
J(p^h,p^h)-\frac{u_0}{\|Jp^h\|}J^2p^h=-({\rm pr}\circ\pi)_*\mathfrak{W}_0+\gamma_1 p^h+\gamma_2 Jp^h.$$
Then
$$\gamma_1=g\left(\frac{2}{\|Jp^h\|}\nabla
J(p^h,p^h)-\frac{u_0}{\|Jp^h\|}J^2p^h, p^h\right).$$
Note that $g(\nabla
J(p^h,p^h), p^h)=\nabla_p^hg(Jp^h, p^h)=0$. So, $\gamma_1=u_0\|Jp^h\|$.

Finally,
$$\gamma_2=\frac{1}{\|Jp^h\|^2}g\left(\frac{2}{\|Jp^h\|}\nabla
J(p^h,p^h)-\frac{u_0}{\|Jp^h\|}J^2p^h, Jp^h\right)$$
Note that since $J$ is antisymmetric, we have $g(J^2p^h, Jp^h)=0$. Therefore, $\gamma_2=\frac{2}{\|Jp^h\|^3}g(\nabla J(p^h,p^h),Jp^h)$, which completes the proof of item (1).
\vspace{1mm}


(2) Relations in this item are direct consequences of relations \eqref{R1} and \eqref{R2} respectively.

\end{proof}

\section{Curvature maps via the Riemannian curvature tensor and the tensor $J$ on $\widetilde M$}
\setcounter{equation}{0}
\setcounter{theor}{0}
\setcounter{lemma}{0}
\setcounter{prop}{0}

Let $\lambda=(p, q),\ q\in M,\ p\in T^*_qM $ be the given
$D$-regular point, as before. Fix $v\in \mathcal V_c(\lambda)$.
As before, denote by $R^\nabla$ the Riemannian curvature tensor.
\begin{theor}\label{main1}
The curvature map $\mathfrak R_\lambda(c, c)$ can be represented as
follows
\begin{equation*}
g\bigg(\big(\mathfrak{R}_\lambda(c,c)(v)\big)^h, v^h\bigg)=g(R^{\nabla}(p^h, v^h)p^h, v^h)
+u_0g(v^h, \nabla J(p^h,v^h))+\frac{u_0^2}{4}\|Jv^h\|^2-\frac{1}{4}\mathcal A^2(\lambda, v),
\end{equation*}
where $\mathcal A$ is as in \eqref{A}
\end{theor}
\begin{proof}
Take $v\in\mathcal V_c(\lambda)$ and parallel vector fields $V$ such that
$V(\lambda)=v$. As in the proof of Lemma \ref{anti} we can take $V$ such that
\begin{equation}\label{commute}
[(Jp^h)^v, V](\bar\lambda)=0,\quad\bar\lambda\in U\cap T_q^*M,
\end{equation}
where $U$ is a neighborhood of $\lambda.$ For simplicity denote $\bar\sigma=(I\circ \rm{PR}^*)\tilde\sigma$.

Recall that by Proposition \ref{preli}, (relation \eqref{R1} there)
$$g\big((\mathfrak R_\lambda(c,c)v)^h,w^h\big) =-\sigma(\hbox{ad}\vec h\ \nabla^c_{V^h},\nabla^c_{v^h}).$$

Let us simplify the right-hand side of the last identity.
First, from the last line of the structural equations \eqref{structeq1} it follows that
\begin{equation}\label{simp1}
\pi_*(\hbox{ad}\vec h(\nabla^c_{V^h}))\in \mathbb R \vec h.
\end{equation}
Then from \eqref{nablac} it follows that
\begin{equation}
\label{simp2}
\sigma(\hbox{ad}\vec h(\nabla^c_{V^h}), \nabla^c_{v^h})=
\sigma(\hbox{ad}\vec h(\nabla^c_{V^h}), \underline{\nabla_{v^h}})
\end{equation}
Further, from the decomposition \eqref{decompeq} it follows that the form $u_0\pi^*d\omega_0=\sigma-\bar\sigma$ is semi-basic (i.e. its interior product with any vertical vector field is zero). Besides, since $v\in \mathcal V_c(\lambda)$, from \eqref{Vc0} it follows that $\pi^*d\omega_0(\vec h, \underline{\nabla_{v^h}})=g(Jp^h, v^h)=0$. Therefore,
\begin{equation}
\label{Rcprel}
g\big((\mathfrak R_\lambda(c,c)v)^h,v^h\big) =-\bar\sigma(\hbox{ad}\vec h\ \nabla^c_{V^h},\nabla_{v^h}).
\end{equation}
Also, from relation \eqref{R1} it follows that it is enough to consider $\hbox{ad}\vec h\ \nabla^c_{V^h}$ modulo
$\mathcal V_a(\lambda)\oplus \mathcal V_b(\lambda)$.

We also need the following
\begin{lemma}
\label{sigg1} Let $V, W $ be vector fields of  $T^*M$ such that
$\pi_*V=\pi_*W =0$. Then
\begin{enumerate}
\vspace{0.5mm}
\item $([(Jp^h)^v, (JV^h)^v])^h=J([(Jp^h)^v, (V^h)^v])^h$.
\vspace{0.5mm}
\item $\sigma([(Jp^h)^v, \underline{\nabla_{V^h}}], \underline{\nabla_{W^h}})=-g(W^h, \nabla J(p^h, V^h)).$
\end{enumerate}
\end{lemma}
\begin{proof}
(1)
It is clear that if item (1) holds for vector field $V$ then also holds for vector field $aV$. Thus in order to prove item (1) it is sufficient to prove
it when $V$ is constant on the fibers of $T^*M$, i.e., when  $V^h$ is a vector field on $\widetilde M$. But in this case from item 1 of Proposition \ref{basic}  for $K=1, N=0$ it follows that both sides of the formula of our item 1 are equal to $-J^2v^h$.

\medskip
(2) Both sides are linear on vector field $V$, thus it is sufficient to prove it when $V$ is constant on the fibers of $T^*M$, which is a direct consequence
of identity \eqref{const1} and Lemma \ref{sigg}.
\medskip
\end{proof}

Now we are ready to start our calculations:
\begin{eqnarray}
\label{simp3}
\hbox{ad}\vec h(\nabla^c_{V^h})&=&
[\underline{\nabla_{p^h}}, \underline{\nabla_{V^h}}]-u_0[(Jp^h)^v, \underline{\nabla_{V^h}}]-\frac{\mathcal A(\lambda,v)}{2\|Jp^h\|}[\underline{\nabla_{p^h}}, (Jp^h)^v]\\
\notag &&-\frac{u_0}{2}[\underline{\nabla_{p^h}}, (JV^h)^v]
+\frac{u_0^2}{2}[(Jp^h)^v, (JV^h)^v],\quad {\rm mod}\, \mathcal V_a(\lambda)\oplus \mathcal V_b(\lambda)
\end{eqnarray}
Note that the last term of \eqref{simp3} vanishes by item (1) of Lemma \ref{sigg1} and relation \eqref{commute}.
Therefore, by \eqref{Rcprel},
\begin{equation}
\label{Rcprel1}
\begin{split}
~&g\big((\mathfrak R_\lambda(c,c)v)^h,v^h\big)=-\bar\sigma([\underline{\nabla_{p^h}}, \underline{\nabla_{V^h}}],\underline{\nabla_{v^h}})+u_0\bar\sigma([(Jp^h)^v, \underline{\nabla_{V^h}}],\underline{\nabla_{v^h}})+\\
~&\frac{\mathcal A(\lambda,v)}{2\|Jp^h\|}\bar\sigma([\underline{\nabla_{p^h}}, (Jp^h)^v],\underline{\nabla_{v^h}}))+\frac{u_0}{2}\bar\sigma([\underline{\nabla_{p^h}}, (JV^h)^v], \underline{\nabla_{v^h}})
\end{split}
\end{equation}

Now we analyze the right-hand side of the last equation term by term.
First, it follows from identity \eqref{Riemnabla} that
\begin{equation}
\label{simp4}
\bar\sigma([\underline{\nabla_{p^h}}, \underline{\nabla_{V^h}}],\underline{\nabla_{v^h}})=-g(R^\nabla(p^h,v^h)p^h,v^h).
\end{equation}
Also it follows from item (2) of Lemma \ref{sigg1} that
 \begin{equation}
\label{simp5}
\bar\sigma([(Jp^h)^v, \underline{\nabla_{V^h}}],\underline{\nabla_{v^h}})=g(\nabla J(p^h,v^h),v^h).
\end{equation}
Also it follows from identity \eqref{Fb1} that
\begin{equation}
\label{simp6}
\bar\sigma([\underline{\nabla_{p^h}}, (Jp^h)^v],\underline{\nabla_{v^h}}))=g(v^h,\nabla J(p^h,p^h)).
\end{equation}
To analyze the fourth term of \eqref{Rcprel1} we need the following

\begin{lemma}
The following identity holds:
\begin{equation}
\label{Riesimp}
\pi_*([\underline{\nabla_{p^h}}, \underline{\nabla_{v^h}})])=\frac{u_0}{2}(Jv^h)^v-\frac{1}{2}\mathcal A(\lambda, v)\frac{(Jp^h)^v}{\|Jp^h\|}\quad {\rm mod}\ \mathbb Rp^h.
\end{equation}
\end{lemma}
\begin{proof}
First, it follows from the equations \eqref{structRiem1} and the identity \eqref{nablA} that $\pi_*([\underline{\nabla_{p^h}}, \underline{\nabla_{v^h}})])=-\widetilde Bv^h$, where $\widetilde B$ is as in \eqref{nablA}. Further, comparing identities \eqref{seria1} and \eqref{nablac}, we get $\widetilde B(v^h)^v=-\frac{u_0}{2}(Jv^h)^v+\frac{1}{2}\mathcal A(\lambda, v)\frac{(Jp^h)^v}{\|Jp^h\|}$. The proof of the proposition is completed.
\end{proof}

Finally, it follows from identity \eqref{Riesimp} that
\begin{equation}
 \label{simp8}
\bar\sigma([\underline{\nabla_{p^h}}, (JV^h)^v], \underline{\nabla_{v^h}})=\bar\sigma(\pi_*([\underline{\nabla_{p^h}}, \underline{\nabla_{v^h}}]), (Jv^h)^v)=g(\frac{u_0}{2}(Jv^h)^v-\frac{1}{2}\mathcal A(\lambda, v)\frac{(Jp^h)^v}{\|Jp^h\|}, Jv^h).
\end{equation}

 Substituting identities \eqref{simp4}, \eqref{simp5}, \eqref{simp6}, and \eqref{simp8} into \eqref{Rcprel1}, we get the required expression for $\mathfrak R_\lambda(c,c)$.
\end{proof}

\begin{theor}
\label{main2}
The curvature maps $\mathfrak R_\lambda(c, b)$  and $\mathfrak R_\lambda(c, b)$ can be represented as
follows
\begin{eqnarray*}
\mathbf{1)}
&&\mathfrak{R}_\lambda(c,b)v=\rho_\lambda(c,b)(v)
\mathcal E_b(\lambda),\ \hbox{where}\ \rho_\lambda(c,b)\in \mathcal V_c(\lambda)^*\ \hbox{and it satisfies}\\
&&\rho_\lambda(c,b)(v)=\frac{1}{\|Jp^h\|}g(R^{\nabla}(p^h,Jp^h)p^h,v^h)-\frac{3}{\|Jp^h\|}g(v^h,\nabla^2J(p^h,p^h,p^h))\\
&+&\frac{4u_0}{\|Jp^h\|}g(v^h,\nabla J(Jp^h,p^h)+\nabla J(p^h,Jp^h))+\frac{u_0^2}{\|Jp^h\|}g(Jv^h,J^2p^h)\\
&+&\frac{8}{\|Jp^h\|^3}g(Jp^h,\nabla J(p^h,p^h))g(v^h,\nabla J(p^h,p^h))-\frac{4u_0}{\|Jp^h\|^3}g(Jp^h,\nabla J(p^h,p^h))g(v^h,J^2p^h);\\
\end{eqnarray*}
\begin{eqnarray*}
\mathbf{2)}
&&
\mathfrak{R}_\lambda(b,b)
\mathcal E_b(\lambda)
=\rho_\lambda(b,b)
\mathcal E_b(\lambda),\
\hbox{where}\\
&&\rho_\lambda(b,b)=\frac{1}{\|Jp^h\|^2}g(R^{\nabla}(Jp^h,p^h)Jp^h,p^h)-\frac{10}{\|Jp^h\|^4}g^2(\nabla
J(p^h,p^h),Jp^h)
\\
&+&\frac{6}{\|Jp^h\|^2}\|\nabla J(p^h,p^h)\|^2+\frac{3}{\|Jp^h\|^2}g(Jp^h,\nabla^2J(p^h,p^h,p^h))-\frac{2u_0}{\|Jp^h\|^2}g(Jp^h,\nabla J(p^h,Jp^h))\\
&-&\frac{3u_0}{\|Jp^h\|^2}g(Jp^h,\nabla
J(Jp^h,p^h))-\frac{6u_0}{\|Jp^h\|^2}g(J^2p^h,\nabla
J(p^h,p^h))+\frac{u_0^2}{\|Jp^h\|^2}\|J^2p^h\|^2
\end{eqnarray*}
\end{theor}

{\bf Sketch of the proof.}
Recall that by Proposition \ref{preli} (relations \eqref{R2} and \eqref{R4} there)
\begin{equation}\label{cbbb}
\begin{split}
~&\rho_\lambda(c,b)v=\sigma( \hbox{ad}\vec h\ \mathcal F_b(\lambda),\nabla^c_{v^h})
\\
~&\rho_\lambda(b,b)=-\sigma( \hbox{ad}\vec h\ \mathcal F_b(\lambda),\mathcal F_b(\lambda)).
\end{split}
\end{equation}

First it follows from \eqref{Fb} that
\begin{equation}
\label{cbbb1}
\begin{split}
~&\hbox{ad}\vec h\ \mathcal F_b(\lambda)=\frac{1}{\|Jp^h\|}(\hbox{ad}\vec h)^2(Jp^h)^v+3\vec h(\frac{1}{\|Jp^h\|})(\hbox{ad}\vec h)(Jp^h)^v\\
~&+3(\vec h)^2(\frac{1}{\|Jp^h\|})(Jp^h)^v+(\vec h)^3(\frac{1}{\|Jp^h\|})\partial_{u_0}
\end{split}
\end{equation}
 Note that the last two terms of \eqref{cbbb1} belong to the space $\mathcal V_a\oplus \mathcal V_b$, which lies in the skew-symmetric complement of
 $\nabla^c_{v^h}\in \mathcal V_c^{\rm trans}$ w.r.t. $\sigma$. Therefore
 \begin{equation}\label{cbbb2}
\rho_\lambda(c,b)v=\sigma\left(\frac{1}{\|Jp^h\|}(\hbox{ad}\vec h)^2(Jp^h)^v-3\vec h(\frac{1}{\|Jp^h\|})(\hbox{ad}\vec h)(Jp^h)^v,\nabla^c_{v^h}\right)
\end{equation}
In a similar way, since $\mathcal V_a=\mathbb R \partial_{u_o}$, we have $\sigma(\partial_{u_0},\mathcal F_b(\lambda))=0$. Therefore
\begin{equation}\label{cbbb3}
\rho_\lambda(b,b)=-\sigma(\frac{1}{\|Jp^h\|}(\hbox{ad}\vec h)^2(Jp^h)^v-3\vec h(\frac{1}{\|Jp^h\|})(\hbox{ad}\vec h)(Jp^h)^v-3(\vec h)^2(\frac{1}{\|Jp^h\|})(Jp^h)^v,\mathcal F_b(\lambda)).
\end{equation}

Note that $(\hbox{ad}\vec h)(Jp^h)^v$ is computed in \eqref{Fb1} and $(\vec h)^2(\frac{1}{\|Jp^h\|})$ is computed in item (2) of Lemma \ref{Jp}.
Furthermore, from  relations \eqref{Fb1} and \eqref{srHam}, using items (1), (2), and (3) of Proposition \ref{basic}, it follows  that
\begin{equation}\label{cbbb4}
\begin{split}
~&(\hbox{ad}\vec h)^2(Jp^h)^v=[\underline{\nabla_{p^h}}-u_0(Jp^h)^v, -\underline{\nabla_{(Jp^h)^v}}+(\nabla
J(p^h,p^h))^v]\\
~&=-2\underline{\nabla_{\nabla J(p^h,p^h)}}+u_0\underline{\nabla_{J^2p^h}}+\|Jp^h\|^2\vec {u}_0-(R^\nabla(p^h,Jp^h)p^h)^v+
\nabla J^2(p^h,p^h,p^h)\\
~&-u_0(\nabla J(Jp^h, p^h))^v-2u_0(\nabla J(p^h, Jp^h))^v+u_0(J\nabla J(p^h,p^h))^v
\end{split}
\end{equation}

Substituting all this into \eqref{cbbb2} and \eqref{cbbb3} and using  identity \eqref{decompeq} and Proposition \ref{basic} one can get
both items of the theorem by long but straightforward computations.
$\Box$
\medskip
Further, let $\mathfrak V_1$ be as in Step 4 of subsection \ref{Imp}. Note that the expression for $\mathfrak V_1^h$ can be found in item (2) of
Lemma \ref{V0W0}.
\begin{theor}\label{main3}
The curvature maps $\mathfrak R_\lambda(c, a)$ and $\mathfrak
R_\lambda(a, a)$ can be represented as follows
\begin{eqnarray*}
\mathbf{1)}\ \mathfrak{R}_\lambda(c, a)v&=&\rho_\lambda(c,a)(v)\frac{\partial_{u_0}}{\|Jp^h\|},\ \hbox{where}\ \rho_\lambda(c,a)\in\mathcal V_c(\lambda)^*\ \hbox{and it satisfies}\\
\rho_\lambda(c,a)v&=&\|Jp^h\|\left(\frac{1}{\|Jp^h\|}\mathcal A\right)^{(2)}(\lambda, v)-g\bigg((\mathfrak R_\lambda(c,c)v\big)^h,
\mathfrak V_1^h\bigg)+\|Jp^h\|\vec h(\frac{1}{\|Jp^h\|})\rho_\lambda(c,b)v \\
\mathbf{2)}\ \mathfrak{R}_\lambda(a, a)\partial_{u_0}&=&\rho_\lambda(a,a)\partial_{u_0},\ \hbox{where}\ \rho_\lambda(c,a)\in\mathcal V_c(\lambda)^*\ \hbox{and it satisfies}\\
\rho_\lambda(a,a)&=&
\vec h\left(\rho_\lambda(c,b)(\mathfrak V_1^h)\right)+\|Jp^h\|
\vec h\left(\frac{1}{\|Jp^h\|}\right)
\vec h(\rho_\lambda(b,b))+\rho_\lambda(c,a)(\mathfrak V_1)\\
&-&\|Jp^h\|
\vec h\left(\frac{1}{\|Jp^h\|}\right)\rho_\lambda(c,b)(\mathfrak V_1)+\|Jp^h\|
\vec h^2\left(\frac{1}{\|Jp^h\|}\right)\rho_\lambda(b,b)+\|Jp^h\|
\vec h^4\left(\frac{1}{\|Jp^h\|}\right)
\end{eqnarray*}
where $\rho_\lambda(c,b)$ and $\rho_\lambda(b,b)$ are as in Theorem
\ref{main2}, $\mathcal A$ is expressed in \eqref{A} and $\mathfrak W_1^h$ is expressed by item (1) of Lemma \ref{V0W0}.
\end{theor}
{\bf Proof}: {\bf 1)}
Recall that by Proposition \ref{preli}, (relation \eqref{R3} there)
\begin{equation}\label{ca1}
\rho_\lambda(c,a)v=\sigma( \hbox{ad}\vec h\nabla^c_{V^h}, \mathcal F_a(\lambda))
\end{equation}
Since $\mathcal E_a(\lambda)$ lies in the skew-symmetric complement of $\mathcal F_c^{\rm trans}(\lambda)$ w.r.t. $\sigma$, then it follows from relations \eqref{F_a} and \eqref{ca1} that
\begin{equation}\label{ca2}
\rho_\lambda(c,a)v=\sigma( \hbox{ad}\vec h\nabla^c_{V^h}, \widetilde{\mathcal F}_a(\lambda))
\end{equation}
Further it follows from relations \eqref{Fa} and \eqref{simp1} that
\begin{equation}\label{ca3}
\rho_\lambda(c,a)v=\sigma( \hbox{ad}\vec h\nabla^c_{V^h}, -\|Jp^h\|\vec
 u_0
-\mathfrak W_0-\|Jp^h\|\vec h\left(\frac{1}{\|Jp^h\|}\right)\mathcal F_b(\lambda))
\end{equation}

Now let us analyze the right-hand side of identity \eqref{ca2} term by term. First from identity \eqref{nablac} it follows that
\begin{equation}\label{ca4}
\sigma( \hbox{ad}\vec h\nabla^c_{V^h}, \vec u_0)=-\vec h(\beta(V))
\end{equation}
Substituting relation \eqref{ca4} into identity \eqref{ca2} and using item (2) of Lemma \ref{V0W0}, we have
\begin{equation}\label{ca5}
\rho_\lambda(c,a)v=-\|Jp^h\|\vec h(\beta(V))-g\bigg((\mathfrak R_\lambda(c,c)v\big)^h,
\mathfrak V_1^h\bigg)+\|Jp^h\|\vec h(\frac{1}{\|Jp^h\|})\rho_\lambda(c,b)v.
\end{equation}
Taking into account item (2) of Lemma \ref{ablem}, we get the item 1) of the theorem.

{\bf 2)} Recall that by Proposition \ref{preli}, (relation \eqref{R5} there)
\begin{equation}\label{aa1}
\rho_\lambda(a,a)=-\sigma( \hbox{ad}\vec h \mathcal F_a(\lambda), \mathcal F_a(\lambda))
\end{equation}
Further, from the fourth line of structural equations \eqref{structeq1} it follows that
\begin{equation}\label{aa2}
\pi_* \hbox{ad}\vec h \mathcal F_a(\lambda)=0,\quad {\rm mod}\ \mathbb Rp^h,\quad \sigma(\hbox{ad}\vec h \mathcal F_a(\lambda), \mathcal F_b(\lambda))=0
\end{equation}
Then it follows from relations \eqref{Fa} and \eqref{F_a} that
\begin{equation}\label{aa3}
\rho_\lambda(a,a)=-\sigma(\hbox{ad}\vec h \mathcal F_a(\lambda),-\|Jp^h\|\vec
 u_0-\mathfrak W_0)
\end{equation}

Now let us analyze the right-hand side of identity \eqref{aa3}. First since $[\vec h, \vec u_0]=0$, we get
\begin{equation}\label{aa4}
\sigma(\hbox{ad}\vec h \mathcal F_a(\lambda),\vec u_0)=
-\vec h(\sigma(\vec u_0, \mathcal F_a(\lambda)))
\end{equation}
Let us calculate $\sigma(\vec u_0, \mathcal F_a(\lambda)).$
Since $$\widetilde{\mathcal F}_a(\lambda)=\vec u_0,\quad{\rm mod}\ \mathcal V_b(\lambda)\oplus\mathcal V_c(\lambda)\oplus\mathcal V^{\rm trans}_b(\lambda)\oplus\mathcal V^{\rm trans}_c(\lambda),$$
we get
\begin{equation}\label{aa6}
\sigma(\vec u_0, \widetilde{\mathcal F}_a(\lambda))=0
\end{equation}
Further, it follows from relation \eqref{F_a} that
\begin{equation}\label{aa7}
\sigma(\vec u_0, \mathcal F_a(\lambda))=-\frac{1}{\|Jp^h\|}\sigma(\hbox{ad}{\vec h}\,\widetilde {\mathcal F}_a(\lambda), \mathcal F_b(\lambda))=-\frac{1}{\|Jp^h\|}\sigma(\hbox{ad}{\vec h}\,\mathcal F_b(\lambda),\widetilde {\mathcal F}_a(\lambda))
\end{equation}

Furthermore, it follows from the line before last of structural equations \eqref{structeq1} and relation \eqref{Fa} that
\begin{equation}\label{aa8}
\sigma\bigl(\hbox{ad}{\vec h}\,\mathcal F_b(\lambda),\widetilde {\mathcal F}_a(\lambda)\bigr)=\sigma\left(\hbox{ad}{\vec h}\,\mathcal F_b(\lambda),-\|Jp^h\|\vec
 u_0-\mathfrak W_0-\|Jp^h\|\vec h\left(\frac{1}{\|Jp^h\|}\right)\mathcal F_b(\lambda)\right)
\end{equation}
Substituting it into \eqref{aa7} and using relation \eqref{cbbb1}, item (2) of Lemma \ref{V0W0} and the second identity of \eqref{cbbb}, we get
\begin{equation}\label{aa9}
\sigma(\vec u_0, \mathcal F_a(\lambda))=-(\vec h)^3(\frac{1}{\|Jp^h\|})-\frac{1}{\|Jp^h\|}\rho_\lambda(c,b)(\mathfrak W_1)+\vec h(\frac{1}{\|Jp^h\|})\rho_\lambda(b,b).
\end{equation}

Finally, we have
\begin{equation}\label{aa10}
\sigma(\hbox{ad}{\vec h}\,\widetilde {\mathcal F}_a(\lambda),\mathfrak W_0)=\sigma\bigl(\hbox{ad}{\vec h}\,\mathfrak W_0,\widetilde {\mathcal F}_a(\lambda)\bigr)=-\rho_\lambda(c,a)\mathfrak W_1.
\end{equation}

Substituting identities \eqref{aa4},\eqref{aa9} and \eqref{aa10} into \eqref{aa3}, we obtain the required expression for $\rho_\lambda(a,a).\quad\quad\quad\quad\quad\quad\quad\quad\quad\quad\quad\quad\quad\quad\quad\quad\quad\quad\quad\quad\quad\quad\quad\quad\quad\quad\quad\quad\quad\quad\quad\quad\quad\quad\quad\Box$

Note that using the calculus developed in the previous section  and the previous theorem, one can express the curvature maps $\mathfrak{R}_\lambda(c, a)$ and $\mathfrak{R}_\lambda(a, a)$  explicitly in terms of the Riemannian metric on $\widetilde M$ and the tensor $J$, but the expressions are too long to be presented here.
Instead we analyze in more detail the expressions for curvature maps
in the case of a uniform magnetic field, i.e. when $\nabla J=0$. Remarkably, the curvature maps $\mathfrak{R}_\lambda(c, a)$ and $\mathfrak{R}_\lambda(a, a)$ vanish in this case.


\begin{cor}\label{uniform}
Assume that $J$ defines a uniform magnetic field , i.e., $\nabla J=0$.
Then the curvature maps have the following form
\begin{enumerate}
\item
 $g\bigg(\big(\mathfrak{R}_\lambda(c,c)(v)\big)^h, v^h\bigg)=g(R^{\nabla}(p^h, v^h)p^h, v^h)+\frac{u_0^2}{4}\Bigl(\|Jv^h\|^2
-\frac{1}{\|Jp^h\|^2}g^2(v^h, J^2p^h)\Bigr);$
\item $\mathfrak{R}_\lambda(c,b)v=\left(\frac{1}{\|Jp^h\|}g(R^{\nabla}(p^h,Jp^h)p^h,v^h)
    +\frac{u_0^2}{\|Jp^h\|}g(Jv^h,J^2p^h)\right)\mathcal E_b(\lambda)
    ;$
\item $\rho_\lambda(b,b)=\frac{1}{\|Jp^h\|^2}g(R^{\nabla}(Jp^h,p^h)Jp^h,p^h)
+\frac{u_0^2}{\|Jp^h\|^2}\|J^2p^h\|^2;$
\item $\mathfrak{R}_\lambda(c, a)=0;$
\item $\mathfrak{R}_\lambda(a, a)=0,$
\end{enumerate}
where $\rho_\lambda(b,b)$ is as in Theorem \ref{main2}.
\end{cor}

{\bf Proof} Items (1), (2) and (3) are direct consequences of Theorems \ref{main1} and \ref{main2}. Now we will show the proofs for items (4) and (5).
We will denote by $X, Y, Z, W, V$ the vector fields on $\widetilde M$. Assume that $v\in\mathcal V_c(\lambda)$ and $V$ is a parallel vector field such that $V(\lambda)=v$. The following two propositions will be needed.
\begin{lemma}\label{Proofun}
 If $\nabla J=0$, then
\begin{enumerate}
 \item For any positive integer $k\in\mathbb N$,\ $\nabla (J^k)=0,\ \nabla^kJ=0$;
 \item $J(R^\nabla(X,Y)Z)=R^\nabla(X,Y)JZ$;
 \item $g(R^\nabla(X,Y)JW, Z)=-g(R^\nabla(X,Y)W, JZ)$;
\end{enumerate}
\end{lemma}
\begin{proof}
The item (1) is proved by definition; The item (2) is an analogy of \cite[Chapter IX, Proposition 3.6 (2)]{knfoundations2}; The item (3) follows from item (2) immediately.
\end{proof}
\begin{lemma}
\label{unia}
For $\forall v\in\mathcal V_c(\lambda)$, the following identities hold:
\begin{enumerate}
\item $\mathcal A(\lambda, v) =-\frac{u_0}{\|Jp^h\|}g(v^h, J^2p^h)$,
\item $\mathcal A^{(1)}(\lambda, v)=\frac{u_0^2}{2\|Jp^h\|}g(v^h, J^3p^h)$,
\item $\mathcal A^{(2)}(\lambda, v)=-\frac{u^3_0\|J^2p^h\|^2}{4\|Jp^h\|^3}
g(v^h, J^2p^h)-\frac{u_0^3}{4\|Jp^h\|}g(v^h, J^4p^h)$.
\end{enumerate}
\end{lemma}
\begin{proof}
The items (1) (2) are direct consequences of Lemma \ref{fun1} and Corollary \ref{A^1}, respectively; The item (3) can be proved by applying Proposition \ref{2tensor} to $\mathcal A^{(1)}$.
\end{proof}
Let us prove $\mathfrak R_\lambda(c,a)=0$. It follows from item (1) of Lemma \ref{Jp} that
\begin{equation}\label{sca1}
\vec h\left(\frac{1}{\|Jp^h\|}\right)=0.
\end{equation}
Then it follows from item 1) of Theorem \ref{main3} that
\begin{equation}\label{sca2}
\rho_\lambda(c,a)v=\mathcal A^{(2)}(\lambda, v)-g\bigg((\mathfrak R_\lambda(c,c)v\big)^h,
\mathfrak V_1^h\bigg)
\end{equation}
Further it follows from item (1) of Lemma \ref{V0W0} that
\begin{equation}\label{sca3}
\mathfrak V_1^h=\frac{u_0}{\|Jp^h\|}J^2p^h+u_0\|Jp^h\|p^h.
\end{equation}
Substituting identity \eqref{sca3} into the expression of $\mathfrak R_\lambda(c,c)$, we get
\begin{equation}\label{sca4}
\begin{split}
~&g\bigg((\mathfrak R_\lambda(c,c)v\big)^h,
\mathfrak V_1^h\bigg)=g(R^{\nabla}(p^h, v^h)p^h, \frac{u_0}{\|Jp^h\|}J^2p^h+u_0\|Jp^h\|p^h)\\
~&+\frac{u_0^2}{4}g\left(J v^h,\frac{u_0}{\|Jp^h\|}J^3p^h+u_0\|Jp^h\|Jp^h\right)\\
~&-\frac{1}{4\|Jp^h\|^2} g(Jv^h, J^2 p^h)g\left(\frac{u_0}{\|Jp^h\|}J^2p^h+u_0\|Jp^h\|p^h, J^2p^h)\right),
\end{split}
\end{equation}
From item (3) of Lemma \ref{unia}  it is easy to see that the sum of the last two items of \eqref{sca4} is equal to $-\mathcal A^{(2)}(\lambda,v)$.
Thus
\begin{equation}\label{sca5}
\rho_\lambda(c,a)v=-g(R^{\nabla}(p^h, v^h)p^h, \frac{u_0}{\|Jp^h\|}J^2p^h+u_0\|Jp^h\|p^h)
\end{equation}
Finally by items (2), (3) of Lemma \ref{Proofun} and algebraic properties of the Riemannian curvature tensor we conclude that $\rho_\lambda(c,a)v=0$.

Now let us prove that $\mathfrak R_\lambda(a, a)=0.$ First using that $R_\lambda(c,a)=0$ and relation \eqref{sca1} we get from item 2) of Theorem \ref{main3} that
\begin{equation}\label{saa1}
\rho_\lambda(a,a)=\vec h\left(\rho_\lambda(c,b)(\mathfrak V_1)\right)
\end{equation}
Let us show that $\rho_\lambda(c,b)(\mathfrak V_1)=0.$ Indeed, from item (2) of the present corollary it follows
\begin{equation}\label{saa2}
\rho_\lambda(c,b)(\mathfrak V_1)=\frac{1}{\|Jp^h\|}g(R^{\nabla}(p^h,Jp^h)p^h,\mathfrak V_1^h)
    +\frac{u_0^2}{\|Jp^h\|}g(J\mathfrak W_1^h,J^2p^h)
\end{equation}
Note that the first term of the right-hand side of last identity coincides with the right-hand side of \eqref{sca5}, taken with the opposite sign. Hence, it vanishes. The second term also vanishes due to relation \eqref{sca3} and the antisymmetricity of $J$.
By this we complete the proof of the Corollary.
$\quad\quad\quad\quad\quad\quad\quad\quad\quad\quad\quad\quad\quad\quad\quad\quad\quad\quad\quad\quad\quad\quad\quad\quad\quad\quad\quad\quad\quad\quad\quad\quad\quad\quad\quad\quad\Box$

Finally consider even more particular but important case when $\nabla J=0$ and $J^2=-{\rm Id}$, i.e. when the tensor $J$ defines a complex structure on $\widetilde M$ and the pair $(g, J)$ defines a K\"{a}hlerian structure on $\widetilde M$. As a direct consequence of the previous theorem, one has
\begin{cor}
\label{Kahler}
Assume that $J$ defines a complex structure on $\widetilde M$, i.e. $\nabla J=0$ and $J^2=-{\rm Id}$. Then
\begin{eqnarray*}
 g((\mathfrak{R}_\lambda(c,c)(v))^h, v^h)&=&g(R^\nabla(p^h, v^h)p^h,v^h)+\frac{u_0^2}{4}\|v\|^2,\\
\mathfrak R_\lambda(b,c)(v)&=&g(R^\nabla(p^h, Jp^h)p^h,v^h)\mathcal E_b(\lambda),\\
\rho_\lambda(b,b)&=&g(R^\nabla(p^h, Jp^h)p^h,Jp^h)+u_0^2,\\
\mathfrak R_\lambda(c,a)&=&0\quad \hbox{and}\quad \mathfrak R_\lambda(a,a)=0,
\end{eqnarray*}
\end{cor}


\section{Comparison Theorems}\label{comparison}
\setcounter{equation}{0}
\setcounter{theor}{0}
\setcounter{lemma}{0}
\setcounter{prop}{0}

In the present section we restrict ourselves to sub-Riemannian structures with a transversal symmetry on a \emph{contact} distribution such that the corresponding tensor $J$ satisfies $\nabla J=0$. We  give estimation of the number of conjugate points (the Comparison Theorem) along the normal sub-Riemannian extremals (Theorem \ref{comparison1} below) in terms of the bounds for the 
curvature of the Riemannian structure on $\widetilde M$  and the tensor $J$.
 The main tool here is the Generalized Sturm Theorem for curves in Lagrangian Grassmannians (\cite{agsymplectic} and \cite{athe}), applied to our structure equation \eqref{structeq1}.

Let, as before, $\lambda=(p,q)\in T^*M, q\in M, p\in T^*_qM$.
Define the following two quadratic forms on the space $\mathcal V_b(\lambda)\oplus\mathcal V_c(\lambda)$
\begin{eqnarray}\label{tildeQ}
\widetilde Q_\lambda(v)&=&\|Jv^h\|^2-\frac{1}{\|Jp^h\|^2}g(Jv^h,Jp^h)^2\\
\label{Q}
Q_\lambda(v)&=&\widetilde Q_\lambda(v)-\frac{3}{4}\widetilde Q_\lambda(v_c),
\end{eqnarray}
where the vector $v_c\in \mathcal V_c(\lambda)$ comes from the decomposition $v=v_b+v_c$  with $v_b\in\mathcal V_b(\lambda)$.
The quadratic form $\widetilde Q_\lambda$ has the natural geometric meaning: the number $\widetilde Q_\lambda(v)$ is equal to the square of the area of the parallelogram spanned by the vectors $Jv^h$ and $Jp^h$ in $T_{{\rm pr}(q)}\widetilde M$ divided by $\|Jp^h\|^2$. In particular,  the quadratic forms $\widetilde Q_\lambda$ are positive definite.
The reason for introducing the form $Q_p$ is that the identities in the Corollary \ref{uniform} can be rewritten as follows, using
 the big curvature map $\mathfrak R_\lambda$ of the sub-Riemannian structure:
\begin{equation}
\label{RQ}
g\Bigl(\bigl(\mathfrak R_\lambda(v)\bigr)^h,v^h\Bigr)=g\bigl(R^\nabla(p^h,v^h)p^h,v^h\bigr) +u_0^2 Q_\lambda(v_{bc}),
\end{equation}
where the vector $v_{bc}\in\mathcal V_b(\lambda)\oplus \mathcal V_c(\lambda)$ comes from the decomposition $v=v_a+v_{bc}$ with $v_a\in \mathcal V_a(\lambda)$.

Now fix $T>0$.
In the sequel given a real analytic function $\varphi :[0, T]\rightarrow \mathbb R$ denote by $\sharp_T\{\varphi(x)=0\}$ the number of zeros of $\varphi$ on the interval $[0, T]$ counted with multiplicities. Given a normal sub-Riemannian extremal $\lambda:[0, T]\rightarrow \mathcal H_{\frac{1}{2}}$ denote by $\sharp_T\bigl(\lambda(\cdot)\bigr)$ the number of conjugate point to $0$ on $(0, T]$.
Let
\begin{equation}
\label{phib}
\phi_\omega (t)=\begin{cases}
\sin\frac{\sqrt{\omega}t}{2}\bigl(\sqrt{\omega}t\cos\frac{\sqrt{\omega}t}{2}-2
\sin\frac{\sqrt{\omega}t}{2}\bigr), & \text{if } \omega\neq 0,\\
t^4& \text{ if } \omega=0
\end{cases};
\end{equation}
\begin{equation}
\label{phic}
\psi_\omega(t)=\begin{cases}
\sin\sqrt{\omega}t, & \text{if } \omega\neq 0,\\
t& \text{ if } \omega=0
\end{cases}.
\end{equation}
Further, define the following
integer valued function on $\mathbb R^2$:

\begin{equation}
\label{ZT1}
Z_T(\omega_b,\omega_c)\stackrel{def}{=}\sharp_T\Bigl\{
\phi_{\omega_b}(t)\psi_{\omega_c}^{n-3}(t)=0\Bigr\}
\end{equation}
An elementary analysis shows that
\begin{equation}
\label{ZT2} Z_T(\omega_b,\omega_c)=
\begin{cases}
(n-3)[\frac{T\sqrt{\omega_c}}{\pi}]+[\frac{T\sqrt{\omega_b)}}{2\pi}]+
\sharp_T\{\tan(\frac{\sqrt{\omega_b}}{2}x)-\frac{\sqrt{\omega_b}}{2}x=0\},
\text{if } \omega_b>0,\ \omega_c>0;\\
[\frac{T\sqrt{\omega_b}}{2\pi}]+\sharp_T\{\tan(\frac{\sqrt{\omega_b}}{2}x)-\frac{\sqrt{\omega_b}}{2}x=0
\},
\quad\text{ if } \omega_b> 0,\ \omega_c\leq 0;\\
(n-3)[\frac{T\sqrt{\omega_c}}{\pi}],\quad \hbox{if}\ \omega_b\leq 0,\ \omega_c> 0.\\
0,\ \ \hbox{if}\ \omega_b\leq 0,\ \omega_c\leq 0.
\end{cases}
\end{equation}
\begin{theor}\label{comparison1}
 Let $\mathfrak c_b$,$\mathfrak c_c$,$\mathfrak C_b$, and $\mathfrak C_c$ are constants such that the curvature tensor $R^\nabla$
 of the Riemannian metric $g$ on $\widetilde M$
satisfies
\begin{equation}
\begin{split}
\label{keyest0}
~&\mathfrak c_b\|v_b^h\|^2+\mathfrak c_c\|v_c^h\|^2\leq g\bigl(R^\nabla(p^h, v_b^h+v_c^h)p^h,v_b^h+v_c^h\bigr)\leq \mathfrak C_b\|v_b^h\|^2+\mathfrak C_c\|v_c^h\|^2,\\
~&\forall \lambda\in \mathcal H_{\frac{1}{2}},v_b\in \mathcal V_b(\lambda), v_c\in\mathcal V_c(\lambda).
\end{split}
\end{equation}
Also let $k_b, k_c, K_b, K_c$ be constants such that
  \begin{equation}
  \label{keyest}
  k_b\|v_b^h\|^2+k_c\|v_c^h\|^2\leq Q_\lambda(v_b+v_c)\leq K_b\|v_b^h\|^2+K_c\|v_c^h\|^2,\quad \forall \lambda\in \mathcal H_{\frac{1}{2}},v_b\in \mathcal V_b(\lambda), v_c\in\mathcal V_c(\lambda).
  \end{equation}
Let $\lambda(\cdot)$ be a normal sub-Riemannian extremal on $\mathcal H_{\frac{1}{2}}\cap\{u_0=\bar u_0\}$
Then the number of conjugate points $\sharp_T\bigl(\lambda(\cdot)\bigr)$ to 0 on $(0,T]$ along $\lambda(\cdot)$ satisfies the following inequality
\begin{equation}
\label{conjest}
Z_T(\mathfrak c_b+k_b\bar u_0^2, \mathfrak c_c+k_c\bar u_0^2)\leq \sharp_T(\lambda(\cdot))\leq   Z_T(\mathfrak C_b+K_b\bar u_0^2, \mathfrak C_c+K_c\bar u_0^2).
\end{equation}
\end{theor}

\begin{remark} If the sectional curvature of the Riemannian metric $g$ on $\widetilde M$ is bounded from below by a constant $\mathfrak c$ and bounded from above by a constant $\mathfrak C$, then in \eqref{keyest0} one can take $\mathfrak c_b=\mathfrak c_c=\mathfrak c$ and
$\mathfrak C_b=\mathfrak C_c=\mathfrak C$. Besides, since $\widetilde Q_\lambda|_{V_b}=Q_\lambda|_{V_b}$,  $\widetilde Q_\lambda|_{V_c}=\frac{1}{4}Q_\lambda|_{V_c}$, and the forms $Q_\lambda$ are positive definite, then the constants $K_b$ and $K_c$
are positive.
\end{remark}

\begin{proof}
We start with some general statements.
Let, as before, $W$ be a linear symplectic space and
$\Lambda:[0,T]\rightarrow L(W)$ be a monotonically nondecreasing curve in the Lagrange Grassmannians $L(W)$ with the constant Young Diagram $D$. In this case  the set of all conjugate points to $0$ is obviously discrete. Denote by $\sharp_T(\Lambda(\cdot))$ the number of conjugate points (counted the multiplicities) of $\Lambda(\cdot)$ on $(0, T]$. Then
$\sharp(\Lambda(\cdot))=\sum_{0<\tau\leq T}\dim(\Lambda(\tau)\cap\Lambda(0))$.
We will use the following corollary of the generalized Sturm theorems from \cite{athe} and \cite{agsymplectic}:
\begin{theor}\label{symplectic}
Let $h_{\tau}, H_{\tau}$ be two quadratic non-stationary Hamiltonians on $W$ such that for any $0\leq \tau\leq T$, the quadratic form $h_{\tau}-H_{\tau}$ is non-positive definite. Let $P_{\tau}, \widetilde P_{\tau}$ be linear Hamiltonian flows generated by $h_{\tau}, H_{\tau}$, respectively:
$$\frac{\partial}{\partial\tau}P_{\tau}=\overrightarrow h_{\tau}P_{\tau},\quad\frac{\partial}{\partial\tau}\widetilde P_{\tau}=\overrightarrow H_{\tau}\widetilde P_{\tau},\quad P_0=\widetilde P_0=id. $$
Further, let $\Lambda(\cdot), \widetilde\Lambda(\cdot)$ be
nondecreasing trajectories of the corresponding flows on
$L(W)$, both having constant Young diagram $D$:
$$\Lambda(\tau)=P_{\tau}\Lambda(0),\quad \widetilde \Lambda(\tau)=\widetilde P_{\tau}\Lambda(0),\quad 0\leq\tau\leq T.$$
Then $\sharp_T(\Lambda(\cdot))\leq \sharp_T(\widetilde{\Lambda}(\cdot)).$
\end{theor}
The detailed proof of this statement (even a in slightly general setting) can be found in \cite{ngeneralized} (see also \cite{azgeometry2}).
As the direct consequence of this theorem and the structural equations \eqref{structeq} we get the following

\begin{cor}
 \label{compcurv}Let $\Lambda, \widetilde \Lambda:[0,T]\rightarrow L(W)$ be two monotonically nondecreasing curves in the Lagrangian Grassmannian $L(W)$ with the same Young diagram $D$. Assume that  $\Lambda(\cdot)$ and $\widetilde \Lambda(\cdot)$ have normal moving frames $(\{E_a(t)\}_{a\in
\Delta}, \{F_a(t)\}_{a\in \Delta})$ and $(\{\widetilde E_a(t)\}_{a\in
\Delta}, \{\widetilde F_a(t)\}_{a\in \Delta})$ respectively such that if $R_t$ is the matrix of the big curvature map of $\Lambda(\cdot)$ w.r.t. the basis $(\{E_a(t)\}_{a\in
\Delta}$ and $\widetilde R_t$ is the matrix of the big curvature map of $\widetilde \Lambda(\cdot)$ w.r.t. the basis $(\{\widetilde E_a(t)\}_{a\in
\Delta}$, then the symmetric matrix $R_t-\widetilde R_t$ is non-positive definite. Then $\sharp_T(\Lambda(\cdot))\leq \sharp_T(\widetilde{\Lambda}(\cdot)).$
\end{cor}

Now let the diagram $D$ be as for the case of sub-Riemannian structures on corank 1 distributions.
Let, as before, $\mathfrak J_\lambda(\cdot)$ is the Jacobi curve attached at the point $\lambda$. Given constants $\omega_b$ and $\omega_c$ let $\Gamma_{\omega_b,\omega_c}(\cdot)$ be the curve in $L(W)$ with the Young diagram $D$ such that its curvature maps satisfy:
\begin{equation}
\label{homcurve}
\mathfrak R_t(a,a)= 0, \mathfrak R_t(c,a)= 0, \mathfrak R_t(c,b)\equiv 0, \mathfrak R_t(b,b) E_b=\omega_b E_b, \mathfrak R_t(c,c)=\omega_c Id\quad  \forall t
\end{equation}
Then from the identity \eqref{RQ}, conditions \eqref{keyest0} and \eqref{keyest},  and Corollary \ref{compcurv} it follows immediately that
\begin{equation}
\label{estprelim1}
\sharp_T\bigl(\Gamma_{\mathfrak c_b+k_b\bar u_0^2, \mathfrak c_c+k_c \bar u_0^2}(\cdot)\bigr)\leq \sharp_T(\mathfrak J_\lambda(\cdot))\leq
\sharp_T\bigl(\Gamma_{\mathfrak C_b+K_b\bar u_0^2, \mathfrak C_c+K_c \bar u_0^2}(\cdot)\bigr)
\end{equation}

In order to prove
Theorem \ref{comparison1} it remains to show that
\begin{equation}
\label{estprelim2}
\sharp_T\bigl(\Gamma_{\omega_b, \omega_c}(\cdot)\bigr)=Z_T(\omega_b,\omega_c).
\end{equation}

Let us prove identity \eqref{estprelim2}.
Let $(E_a(t),E_b(t), E_c(t), F_a(t), F_b(t), F_c(t)$ be a normal moving frame of the curve
 $\Gamma_{\omega_b, \omega_c}(\cdot)$. Substituting \eqref{homcurve} into the structural equation
 \eqref{structeq1} we get
  \begin{equation}
\label{structeq2}
\begin{cases}
E_a^\prime(t)=E_b(t)\\
E_b^\prime(t)=F_b(t)\\
E_c^\prime(t)=F_c(t)\\
F_a^\prime(t)=0\\
F_b^\prime(t)=-\omega_bE_b(t)\mathcal{R}_t(b, b)-F_a(t)\\
F_c^\prime(t)=-\omega_c E_c(t).
\end{cases}
\end{equation}
From this we obtained the following two separated equations for $E_a$ and for $E_c$, respectively:
\begin{equation}
\label{structeq3}
\begin{cases}
E_a^{(4)}+\omega_b E_a''=0\\
E_c''+\omega_c E_c=0
\end{cases}
\end{equation}
Assume first that $\omega_b\neq 0$ and $\omega_c\neq 0$.
Then  there exist vectors $\alpha_1,\ldots,\alpha_4$ and $\beta_1^k,\beta_2^k$, $k=1,\ldots n-3$ in $W$ such that
\begin{equation}
\label{Eabc}
\begin{split}
~&E_a(t)=e^{i\sqrt{\omega_b}t} \alpha_1+e^{-i\sqrt{\omega_b}t}\alpha_2+\alpha_3+t \alpha_4,\\
~&E_b(t)=i\sqrt{\omega_b}e^{i\sqrt{\omega_b}}\alpha_1-i\sqrt{\omega_b}e^{-it\sqrt{\omega_b}}\alpha_2+\alpha_4,\\
~&E_c(t)=\bigl (e^{i\sqrt{\omega_c}t}\beta_1^1+e^{-i\sqrt{\omega_c}t}\beta_2^1,\ldots,e^{i\sqrt{\omega_c}t}\beta_1^{n-3}+
e^{-i\sqrt{\omega_c}t}\beta_2^{n-3}).
\end{split}
\end{equation}
Besides, by constructions vectors $\alpha_1,\ldots,\alpha_4,\beta_1^1,\beta_2^1,\ldots,\beta_1^{n-3},\beta_2^{n-3}$ have to be linearly independent.

Introducing some coordinates in $W$ we can look on the tuple $\bigl(E_a(t), E_b(t), E_c(t), E_a(0), E_b(0), E_c(0)\bigr)$ as on
$2(n-1)\times 2(n-1)-$matrix, representing each involved vector as a column. Let $d(t)$ be the determinant of this matrix. Obviously, $\bar t$ is conjugate point to $0$ of multiplicity $l$ if and only if $\bar t$ is zero of multiplicity $l$ of function $d(t)$. On the other hand, using expressions \eqref{Eabc} it is easy to show that the function $d(t)$ is equal, up to a nonzero constant factor, to
\begin{equation*}
\begin{vmatrix}
e^{i\sqrt{\omega_b}t}&i\sqrt{\omega_b}e^{i\sqrt{\omega_b}t}&1&i\sqrt{\omega_b}\\
e^{-i\sqrt{\omega_b}t}&-i\sqrt{\omega_b}e^{-i\sqrt{\omega_b}t}&1&-i\sqrt{\omega_b}\\
1&0&1&0\\
t&1&0&1
\end{vmatrix}\cdot
\begin{vmatrix}e^{i\sqrt{\omega_c}t}&1\\
e^{-i\sqrt{\omega_c}t}&1
\end{vmatrix}^{n-3},
\end{equation*}
which in turn is equal, up to a nonzero constant factor, to the function $\phi_{\omega_b}(t)\psi_{\omega_c}^{n-3}(t)$
appearing in the definition \eqref{ZT1} of the function $Z_T(\omega_b,\omega_c)$.
The case when one or both $\omega_b$ and $\omega_c$ are equal to zero can be treated analogously.
This completes the proof of \eqref{estprelim2} and Theorem \ref{comparison1} itself.
\end{proof}

Now let us state separately what Theorem \ref{comparison1} says about the intervals along normal extremals of the considered sub-Riemannian structure which do not contain conjugate points or contain at least one conjugate point:

\begin{cor}\label{estimation}
Under the same estimates on the
curvature of the Riemannian metric $g$ on $\widetilde M$ and on the quadratic forms $Q_\lambda$ as in  Theorem \ref{comparison1} the following statement hold for a normal sub-Riemannian extremal on $\mathcal H_{\frac{1}{2}}\cap\{u_0=\bar u_0\}$:

\begin{enumerate}
 \item If $\mathfrak C_b+K_b\bar u_0^2>0$ and $\mathfrak C_c+K_c\bar u_0^2>0$, then there are no conjugate points to $0$ in the interval $\bigl(0, \min\{\frac{2\pi}{\sqrt{\mathfrak C_b+K_b\bar u_0^2}},\frac{\pi}{\sqrt{\mathfrak C_c+K_c\bar u_0^2}}\}\bigr)$;
 \item If $\mathfrak C_b+K_b\bar u_0^2>0$ and $\mathfrak C_c+K_c\bar u_0^2\leq 0$, then there are no conjugate points to $0$ in $(0, \frac{2\pi}{\sqrt{\mathfrak C_b+K_b\bar u_0^2}})$;
\item If $\mathfrak C_b+K_b\bar u_0^2\leq 0$ and  $\mathfrak C_c+K_c\bar u_0^2>0$, then there are no conjugate points to $0$ in $(0, \frac{\pi}{\sqrt{\mathfrak C_c+K_c\bar u_0^2}})$;
\item If $\mathfrak C_b+K_b\bar u_0^2\leq 0$ and $\mathfrak C_c+K_c\bar u_0^2\leq 0$, then there are no conjugate points to $0$ in $(0, \infty)$;
\item If $\mathfrak c_b+k_b\bar u_0^2>4(\mathfrak c_c+k_c\bar u_0^2)>0$,
 then there is at least one conjugate point to $0$ in $(0, \frac{2\pi}{\sqrt{\mathfrak c_b+k_b\bar u_0^2}}]$;
\item If $\mathfrak c_c+k_c\bar u_0^2\geq\frac{1}{4}(\mathfrak c_b+k_b\bar u_0^2)>0$,
then there is at least $n-3$ conjugate points to $0$ in $(0, \frac{\pi}{\sqrt{\mathfrak c_c+k_c\bar u_0^2}}]$ ( at  least $n-2$ conjugate points in the case $\mathfrak c_b+k_b\bar u_0^2=4(\mathfrak c_c+k_c\bar u_0^2)>0$);


 \item If $\mathfrak c_b+k_b\bar u_0^2>0$ and  $\mathfrak c_c+k_c\bar u_0^2\leq 0$, then there is at least one conjugate point to $0$ in $(0, \frac{2\pi}{\sqrt{\mathfrak c_b+k_b\bar u_0^2}}]$
\item If $\mathfrak c_b+k_b\bar u_0\leq 0$ and  $\mathfrak c_c+k_c\bar u_0^2>0$, then there is at least $n-3$ conjugate points to $0$ in $(0, \frac{\pi}{\sqrt{\mathfrak c_c+k_c\bar u_0^2}}]$
\end{enumerate}
\end{cor}

\medskip
Finally note that if in addition $J^2=-{\rm Id}$ then the quadratic forms $Q_\lambda$ have the following simple form:
\begin{equation*}
Q_\lambda(v_c+v_b)=\|v_b^h\|^2+\frac{1}{4}\|v_c^h\|^2 \quad  \forall v_b\in\mathcal V_b(\lambda), v_c\in\mathcal V_c(\lambda).
\end{equation*}

Therefore in this case one can take $k_b=K_b=1$ and $k_c=K_c=\frac{1}{4}$.

\end{document}